\documentclass[11pt, reqno]{amsart}
\usepackage{latexsym, amsmath, amssymb, amsthm, verbatim}
\usepackage{cases}
\usepackage[colorlinks]{hyperref}
\usepackage{hypernat}
\usepackage{xcolor} 
\usepackage{color}

\usepackage[margin=1.5in]{geometry}

\newtheorem{thm}{Theorem}
\newtheorem{lem}[thm]{Lemma}
\newtheorem{prop}[thm]{Proposition}
\newtheorem{cor}[thm]{Corollary}

\theoremstyle{remark}
\newtheorem{rmk}[thm]{Remark}
\newtheorem{example}[thm]{Example}

\theoremstyle{definition}
\newtheorem{defi}[thm]{Definition}

\numberwithin{thm}{section}
\numberwithin{equation}{section}

\def\R{{\mathbb R}}

\newcommand{\vep}{\varepsilon}
\newcommand{\ol}{\overline}
\newcommand{\ul}{\underline}

\newcommand{\bpm}{\begin{pmatrix}}
\newcommand{\epm}{\end{pmatrix}}

\newcommand{\beq}{\begin{equation}}
\newcommand{\eeq}{\end{equation}}

\DeclareMathOperator*{\limsups}{limsup^\ast}

plus 6pt minus 12pt
\title[Singular Dirichlet problems]{\protect{Singular Dirichlet boundary problems for a class of fully nonlinear parabolic equations in one dimension}}

\author{Takashi Kagaya}
\address{Takashi Kagaya, Muroran Institute of Technology, Japan}
\email{kagaya@muroran-it.ac.jp}

\begin{document}

\begin{abstract}
In this paper, we deal with the initial value problem for a class of fully nonlinear parabolic equations with a singular Dirichlet boundary condition in one space dimension. 
The interior equation includes, for example, a fully nonlinear $p$-Laplace type heat equation and a $\beta$-power type curvature flow. 
The singular Dirichlet boundary condition depicts, for example, the asymptoticness of the ends of complete curve to parallel two lines in geometric flow of graphs. 
We study the dependence of the existence and non-existence of solution to the problem on the interior equation and the boundedness of the initial function. 
\end{abstract}

\subjclass[2020]{35A01, 35K61, 35D40}
\keywords{fully nonlinear parabolic equations, singular Dirichlet problem, viscosity solutions, existence and non-existence theory}

\maketitle

\section{Introduction}

This paper is concerned with singular Dirichlet boundary problem for fully nonlinear differential equation. 
Let $b>0$ and $f,g \in C(\mathbb{R})$. 
We are mainly interested in the following equation in one space dimension: 
\begin{numcases}{}
u_t-f(g(u_x) u_{xx})=0 & \text{in $(-b, b)\times (0, \infty)$,} \label{flow eq}\\
\lim_{x \to \pm b} u(x, t)= \infty  & \text{for any $t>0$,} \label{bdry}\\
u(\cdot , 0)=u_0 & \text{in  $(-b, b)$,}  \label{initial}
\end{numcases}
where $u_0$ is a given continuous function in $(-b, b)$ satisfying several assumptions to be elaborated later, and $f$ and $g$ are given continuous function satisfying 
\begin{itemize}
\item[(A1)] $f \in C(\mathbb{R})$ is strictly increasing with $f(0) = 0$ and $f(s)\to \pm \infty$ as $s \to \pm \infty$. 
\item[(A2)] $g \in C(\mathbb{R})$ is positive and 
\beq\label{exponent alpha}
 |s|^{\alpha}g(s) \to C_{g\pm} \quad \text{as } s \to \pm \infty 
 \eeq
for some exponent $\alpha \in \mathbb{R}$ and constants $C_{g+}, C_{g-} > 0$. 
\end{itemize}
One of the examples of the interior equation \eqref{flow eq} is a $p$-Laplace type heat equation 
\begin{equation}\label{p-heat}  
u_t = |\Delta_p u + \varepsilon \Delta u|^{\beta_1-1} (\Delta_p u + \varepsilon \Delta u) \quad \text{in} \; \; (-b, b) \times (0, \infty), 
\end{equation}
where $p \ge 2, \beta_1 > 0$ and $\varepsilon > 0$ are constants. 
Since $\Delta_p u + \varepsilon \Delta u$ can be re-written as 
\[ \Delta_p u + \varepsilon \Delta u = \{(p-1)|u_x|^{p-2} + \varepsilon\} u_{xx}, \]
it is a special case of \eqref{flow eq} with 
\[ f(s) = |s|^{\beta_1 - 1} s, \quad g(s) = (p-1) |s|^{p-2} + \varepsilon \]
for $s \in \mathbb{R}$. 
The function $g$ satisfies (A2) with $\alpha = 2-p$. 
Related to this equation, we refer to \cite{LYCGX, LZ, SL, To} for the existence theory for quasi-linear $p$-Laplace type heat equations $u_t - \Delta_p u = h(x, u, \nabla u)$ in whole space. 
In our problem, we need a perturbation term $\varepsilon \Delta u$, but the $p$-Laplacian may have nonlinear effects $f$. 

Another typical example of \eqref{flow eq} is the following equation: 
\begin{equation}\label{cf-eq} 
u_t = (1+|u_x|^2)^{\frac{1-3\beta_2}{2}} |u_{xx}|^{\beta_2 -1} u_{xx} \quad \text{in} \; \; (-b, b) \times (0, \infty), 
\end{equation}
where $\beta_2 > 0$ is a constant. 
It is a special case of \eqref{flow eq} with 
\[ f(s) = |s|^{\beta_2 -1} s, \quad g(s) = (1 + s^2)^{\frac{1 - 3\beta_2}{2\beta_2}} \]
for $s \in \mathbb{R}$. 
The function $g$ satisfies (A2) with $\alpha = 3 - \frac{1}{\beta_2}$. 
Such an equation has a geometric interpretation: it describes the motion of a graph-like planar curve whose normal velocity equals to the singed $\beta_2$-th power of its curvature. 
The singular Dirichlet boundary condition \eqref{bdry} depicts that the graph $y= u(x,t)$ is a complete curve whose two ends are asymptotic to two parallel lines $x = \pm b$ at each time $t > 0$. 
This kind of signed power type curvature flow of complete curves was studied in \cite{CCD} to prove the well-posedness when the initial curve is the graph of a smooth and strictly convex function and the convergence of the solution to a traveling wave solution. 
Our study is the extension of their existence theory for the general interior equation \eqref{flow eq}. 
We will discuss the difference between their results and our results later. 

The purpose of this paper is to study the existence and non-existence thresholds of solution to \eqref{flow eq}--\eqref{initial} depending on $f, g$ and $u_0$, which is motivated by some facts for the following singular Neumann boundary problem: 
\begin{equation} \label{singular-neu}
\begin{cases}
u_t-f(g(u_x) u_{xx})=0 & \text{in $(-b, b)\times (0, \infty)$,} \\
\lim_{x \to \pm b} u_x(x, t)= \pm \infty  & \text{for any $t>0$,} \\
u(\cdot , 0)=u_0 & \text{in  $[-b, b]$} 
\end{cases} \end{equation}
with the assumptions (A1) and (A2). 
The facts are as follows: 
\begin{itemize}
\item[(I)] Traveling wave solutions formed by $w(x,t) = W(x) + ct$ with $W \in C^2((-b, b))$ and $c \in \mathbb{R}$ uniquely exist up to vertically translations if $\alpha > 1$. Furthermore, $W$ is bounded if $\alpha > 2$ (see \cite[Theorem 1.3]{KL}) and satisfies 
\[ \lim_{x \to \pm b} W(x) = \infty \] 
if $1 < \alpha \le 2$ (see (a) in Proposition \ref{prop:exists-tw}). If $\alpha \le 1$, there is no traveling wave solution (see (b) in Proposition \ref{prop:exists-tw}). 
\item[(II)] If $u_0 \in C([-b, b])$ and $\alpha > 2$, there is a unique viscosity solution $u \in C([-b, b] \times [0, \infty))$ exists. In particular, $u(\cdot, t)$ is bounded on $[-b, b]$ for any $t > 0$ since $u(\cdot, t)$ is of class $C([-b, b])$. Furthermore, $u(\cdot, t)$ converges to a traveling wave solution in $L^\infty([-b, b])$ as $t \to \infty$ under the convexity assumption of $u_0$ and some additional assumptions on $f$ and $g$ (see \cite[Theorem 1.1 and Theorem 1.4]{KL}). 
\item[(III)] If $\alpha \le 2$, there is no viscosity solution $u \in C([-b, b] \times [0, \infty))$. In particular, if we assume the existence, then we can prove that $u(\pm b, t) = \infty$ for any $t>0$, which contradicts to the boundedness of $u(\cdot, t)$ (see \cite[Theorem 1.2]{KL}). 
\end{itemize}

Due to the fact (II), we can see that the traveling wave solutions are stable for the problem \eqref{flow eq}--\eqref{bdry} if $\alpha > 2$. 
Even for the case $1 < \alpha \le 2$, we may expect the stability of the traveling waves. 
Indeed, for the singed power type curvature flow \eqref{cf-eq} with the boundary condition \eqref{bdry}, the previous work \cite{CCD} shows that the bounded traveling waves in case $\beta_2 > 1$ and the unbounded traveling waves in case $\frac{1}{2} < \beta_2 \le 1$ are stable. 
We note that the above cases correspond to $\alpha > 2$ and $1 < \alpha \le 2$, respectively, due to the relation $\alpha = 3 - \frac{1}{\beta_2}$. 
In terms of this stability of the traveling waves, in the case $1 < \alpha \le 2$, solutions to \eqref{flow eq}--\eqref{initial} with bounded $u_0 \in C([-b, b])$ may diverge to $\infty$ at the boundary points $x=\pm b$ as the time $t$ passes. 
The fact (III) shows that the solutions ``instantaneously'' blow up at the boundary points $x=\pm b$ in the case $\alpha \le 2$. 

Therefore, a natural question is whether the existence of a unbounded solution to \eqref{flow eq}--\eqref{initial} has a threshold with respect to $\alpha$, as does the existence of the traveling wave solutions to \eqref{singular-neu}, or whether it depends on other conditions. 
In this paper, we will show that the existence of solutions to \eqref{flow eq}--\eqref{initial} also depends on a condition of $f$ and the boundedness of $u_0$ in addition to $\alpha$ in the assumption (A2). 

For the initial function $u_0$, we assume either of the following conditions: 
\begin{itemize}
\item[(B1)] $u_0$ is of class $C([-b, b])$

\item[(B2)] $u_0$ is of class $C((-b, b))$ and satisfies 
\begin{equation}\label{initial div} 
\lim_{x \to \pm b} u_0(x) = \infty, \quad \limsup_{x \to b } u_0(x) (b-x)^\gamma < \infty, \quad \limsup_{x \to -b} u_0(x) (b+x)^\gamma < \infty 
\end{equation}
for some $\gamma > 0$

\item[(B3)] $u_0$ is of class $C((-b,b))$ and satisfies
\begin{equation}
\lim_{x \to b} \left(u_0(x) - D_+ \psi_{\gamma_+} (b-x)\right)=\hat{C}_+ \; \; \text{and} \; \; \lim_{x \to -b} \left(u_0(x) - D_-\psi_{\gamma_-}(b+x)\right)=\hat{C}_- \label{as-order-1}
\end{equation}
for some constants $\gamma_\pm \ge 0, D_\pm >0$ and $\hat{C}_\pm \in \mathbb{R}$, where 
\begin{equation}\label{def-psi} 
\psi_\gamma (s) := 
\begin{cases} 
s^{-\gamma} & \text{if} \quad \gamma > 0, \\
- \log s & \text{if} \quad \gamma = 0. 
\end{cases}
\end{equation}
\end{itemize}
We note that (B1) yields the boundedness of $u_0$. 
(B3) is a stronger assumption than (B2), which will be needed to prove the uniqueness of the solution. 

We here give remarks on some difficulties to establish the existence theory. 
As is easily seen, there are two major difficulties caused by the structure of the equation: high nonlinearity of the operator due to the appearance of $f$ and strong degeneracy of the elliptic operator because of the rapid decay of $g(s)$ as $s \to \infty$ when $\alpha > 0$. 
One therefore cannot expect solutions to be smooth in these circumstances. 
The framework of viscosity solutions (cf.\ \cite{CIL}) then becomes a natural choice to study \eqref{flow eq}. 
Singular Dirichlet and Neumann problems for semi-linear elliptic equations on a bounded domain $\Omega \subset \mathbb{R}^n$ are studied in \cite{LaL2} with the boundary condition interpreted in the viscosity sense: $u- \phi$ achieves its minimum only at interior points within $\Omega$ for any $C^2(\overline{\Omega})$ function $\phi$. 
The viscosity boundary condition is a weak form of two singular boundary conditions, whereas we only need to deal with the singular Dirichlet boundary condition in our problem. 
We thus adopt the boundary condition \eqref{bdry} for the framework of viscosity solutions in the classical sense. 
A precise definition of viscosity solutions is given in Section \ref{sec:viscosity}. 

We will use comparison arguments to prove the uniqueness and adopt Perron's method (cf.\ \cite{CIL}) when we prove the existence of the solution. 
We remark that the well-posedness was established in \cite{AtBa, BaDa1, BaDa2} for generalized Dirichlet boundary value problems and in \cite{KL} for the singular Neumann boundary problem. 
Our approach is based on their arguments, especially in the proof of comparison principle. 
However, since $\sup_{x \in (-b,b)} \phi(x) - \psi(x)$ is not always a bounded value for two functions $\phi, \psi: (-b, b) \to \mathbb{R}$ which diverge to $\infty$ at the boundary points $x=\pm b$, their arguments can not be expanded directly to our problem. 
Roughly speaking, to overcome this difficulty, we introduce inequalities
\begin{align}
&u_t \ge \max\{f((1+\delta) g(u_x) u_{xx}), 0\} \quad \text{in} \; \; (-b, b) \times (0, \infty), \label{scale eq1}\\
&u_t \le \min\{f(g(u_x) u_{xx}), 0\} \quad \text{in} \; \; (-b, b) \times (0, \infty) \label{scale eq2}
\end{align}
for $\delta > 0$. 
Since the introduction enable us to apply the scaling argument easily, we can discuss the comparison between (i) sub-solution to \eqref{scale eq2} with \eqref{bdry} and super-solution to \eqref{flow eq} and \eqref{bdry}; or (ii) sub-solution to \eqref{flow eq} and \eqref{bdry} and super-solution to \eqref{scale eq1} with \eqref{bdry}. 
Via constructing good super- and sub-solutions to \eqref{scale eq1} and \eqref{scale eq2}, this comparison argument yields that the divergence rate \eqref{as-order-1} is preserved in time for special $\gamma_\pm \ge 0$ if $\alpha > 1$, namely, $u(\cdot, t)$ also satisfies the condition \eqref{as-order-1}. 
Due to this preservation, we can establish a comparison theory for \eqref{flow eq}--\eqref{bdry} under the assumption (B3), which implies the uniqueness of the solution. 

Final remark on difficulties of our problem is that the existence or non-existence theory depends on $f, \alpha$ and $u_0$. 
This structure can be seen when we will construct super-solutions to adopt Perron's method in the existence theory and when we will construct a sequence of sub-solutions to prove the non-existence theory. 
Heuristically arguments to construct the super- and sub-solutions are noted in Section \ref{sec:exist}. 

We first present the existence theory when $\alpha > 2$. 
In this case, we can prove the existence of a solution to \eqref{flow eq}--\eqref{initial} if we assume the divergence property of $u_0$ at the boundary points as follows:

\begin{thm}\label{thm:existence}
Let $b>0$. 
Assume that $f, g$ and $u_0$ respectively satisfy (A1), (A2) with $\alpha > 2$. 
If $u_0$ satisfies (B2), then there exists a viscosity solution $u: (-b, b) \times [0, \infty) \to \mathbb{R}$ to \eqref{flow eq}--\eqref{initial}. 
If $u_0$ satisfies (B3) with $\gamma_\pm \ge 0$, then the viscosity solution is continuous and unique. 
\end{thm}

On the other hand, if $1 < \alpha \le 2$, the problem \eqref{flow eq}--\eqref{initial} has a solution independent of the boundedness of $u_0$ as follows: 

\begin{thm}\label{thm:existence2}
Let $b>0$. 
Assume that $f, g$ and $u_0$ respectively satisfy (A1), (A2) with $1 < \alpha \le 2$. 
If $u_0$ satisfies (B1) or (B2), then there exists a viscosity solution $u: (-b, b) \times [0, \infty) \to \mathbb{R}$ to \eqref{flow eq}--\eqref{initial}. 
If $u_0$ satisfies (B3) with $\gamma_\pm \ge \frac{2-\alpha}{\alpha -1}$, then the viscosity solution is continuous and unique. 
\end{thm}

In the case $\alpha \le 1$, the existence and non-existence of the solution for the singular Dirichlet boundary problem depend on the divergence rate of $f(s)$ as $s \to \pm \infty$. 
Therefore, we assume the following additional condition for $f$: 
\begin{itemize}
\item[(A3)] $f$ satisfies 
\[ |s|^{-\beta} f(s) \to \pm C_{f\pm} \quad \text{as} \; \; s \to \pm \infty \]
for some $\beta > 0$ and constants $C_{f+}, C_{f-} > 0$. 
\end{itemize}
A typical example of $f$ satisfying (A1) and (A3) is $f(s) = |s|^{\beta-1} s$. 
Then, we have the following result. 

\begin{thm}\label{thm:non-existence}
Let $b>0$. 
Assume that $f$ satisfies (A1) and (A3), and assume also $g$ satisfies (A2) with $\alpha \le 1$. 
\begin{itemize}
\item[(a)] If $\alpha < 1$, $\beta \ge \frac{1}{1-\alpha}$ and $u_0 \in C((-b, b))$ satisfies $\inf_{x \in (-b, b)} u_0(x) > -\infty$, then there is no viscosity solution $u: (-b, b) \times [0, \infty) \to \mathbb{R}$ to \eqref{flow eq}--\eqref{initial}. 
\item[(b)] If $\alpha = 1$ or $\alpha < 1$ and $\beta < \frac{1}{1-\alpha}$, assume $u_0$ satisfies (B1) or (B2). Then there exists a viscosity solution $u: (-b, b) \times [0, \infty) \to \mathbb{R}$ to \eqref{flow eq}--\eqref{initial}. 
\end{itemize}
\end{thm}
In the proof of (a) in Theorem \ref{thm:non-existence}, we prove that $u(x,t)= \infty$ at any points $x \in (-b,b)$ and any time $t>0$ if a viscosity solution exists, which yields formally that ``instantaneous interior blow-up'' occurs. 
The proof suggests that the divergence of the solution to infinity at the boundary points propagates rapidly to the interior due to the diffusion effects of the interior equation. 

As applications of these theorems, we summarize the existence and non-existence of solutions to the singular Dirichlet boundary problems for the typical examples \eqref{p-heat} and \eqref{cf-eq}. 
First corollary is for \eqref{p-heat}. 

\begin{cor}
Assume $b>0$. 
Assume also $p \ge 2, \beta_1 > 0$ and $\varepsilon > 0$. 
\begin{itemize}
\item If $\beta_1 \ge \frac{1}{p-1}$ and $u_0 \in C((-b, b))$ satisfies $\inf_{x \in (-b, b)} u_0 > -\infty$, then there is no viscosity solution to \eqref{bdry}, \eqref{initial} and \eqref{p-heat}. 
\item If $0 < \beta_1 < \frac{1}{p-1}$ and $u_0$ satisfies (B1) or (B2), then a viscosity solution to \eqref{bdry}, \eqref{initial} and \eqref{p-heat} exists globally in time. 
\end{itemize}
\end{cor}
Since \eqref{p-heat} coincides with the heat equation $u_t = (1+\varepsilon)u_{xx}$ if $\beta_1 = 1$ and $p=2$, the corollary shows that the heat equation with singular Dirichlet boundary condition does not have any viscosity solutions. 

For the signed power type curvature flow \eqref{cf-eq} with the singular Dirichlet boundary condition \eqref{bdry}, the existence theory in \cite{CCD} requires that $u_0$ is smooth and strictly convex in $(-b, b)$ and satisfies 
\[ \lim_{x \to \pm b} u_0(x) = \infty. \]
Our existence theory does not need the smoothness, convexity and, when $\beta_2 \le 1$, the unboundedness of $u_0$ although an additional assumption is necessary to obtain the uniqueness. 

\begin{cor}
Assume $b>0$ and $\beta_2 > 0$. 
\begin{itemize}
\item If $\beta_2 > 1$ and $u_0$ satisfies (B2), then a viscosity solution to \eqref{bdry}, \eqref{initial} and \eqref{cf-eq} exists globally in time. 
Furthermore, if $u_0$ satisfies (B3) with $\gamma_\pm \ge 0$, then the solution is continuous and unique. 
\item If $0 < \beta_2 \le 1$ and $u_0$ satisfies (B1) or (B2), then a viscosity solution to \eqref{bdry}, \eqref{initial} and \eqref{cf-eq} exists globally in time. 
Furthermore, if $\frac{1}{2} < \beta_2 \le 1$ and $u_0$ satisfies (B3) with $\gamma_\pm \ge \frac{1-\beta_1}{2\beta_2 - 1}$, then the solution is continuous and unique. 
\end{itemize}
\end{cor}

We now refer to another previous studies related to our problem. 
As mentioned earlier, the singular Dirichlet boundary condition depicts the asymptoticness of the ends of complete curve and surface respectively to parallel two lines or a cylinder in geometric flow of graphs. 
We refer to \cite{ChDa, CDKL, WaWo} in addition to \cite{CCD} for the studies on geometric flows of complete curves/surfaces with this kind of condition of ends. 
In particular, Choi-Daskalopoulos-Kim-Lee \cite{CDKL} showed the global-in-time existence of graph solution to the initial value problem for a power type Gaussian curvature flow, which is a high dimensional problem of \eqref{cf-eq} with \eqref{bdry}. 
These results will provide useful insight into the existence and non-existence thresholds for solutions to a higher dimensional version of our problem. 

Since our problem has traveling wave solutions when $\alpha > 1$, we may expect that our solutions constructed above converges to the traveling wave solutions as $t \to \infty$ in this case. 
Related to the convergence results to traveling wave solutions, we also refer to \cite{AlWu2, AlWu} for the contact angle problem in the context of capillary curves/surfaces, \cite{BPT, Th} for the Dirichlet problem for viscous Hamilton-Jacobi equations and \cite{CeNo} for the mean curvature flow with periodic forcing. 
In particular, Barles-Porretta-Tchamba \cite{BPT} and Cesaroni-Novaga \cite{CeNo} studied the convergence of solution starting from a bounded initial data to unbounded traveling wave solution. 
In our problem, long time behavior analysis, including the case $\alpha \le 1$ where the asymptotic behavior of solutions is nontrivial, should be a future work. 

As non-existence theory for solutions to semi-linear parabolic equations, we also refer to \cite{FuIo, IRT} for problems in whole space and \cite{BrCa, IRT, LRSV, We} for zero Dirichlet boundary value problems. 
For another example, the non-existence theory has recently been discussed for semi-linear parabolic equations on manifolds (cf.\ \cite{TaYa, WeZh}) and for fractional parabolic equations (cf.\ \cite{FHIL, Su}). 
We prove the non-existence theory by using the propagation into the interior of the divergence of solutions at boundary points due to the diffusion effect of the interior equation, whereas the above results discuss the non-existence of solutions depending on the singularity of initial functions.

The rest of the paper is organized in the following way. 
We first introduce the definition of viscosity solutions of the singular Dirichlet boundary problems in Section \ref{sec:viscosity}. 
Section \ref{sec:comparison} is devoted to the proof of the comparison principle. 
In Section \ref{sec:exist}, via heuristically arguments to construct super- and sub-solutions to prove the main results, we give rigorous proofs of Theorem \ref{thm:existence}--\ref{thm:non-existence}. 

\subsection*{Acknowledgments} 
The author is grateful to Prof.\ Qing Liu (Okinawa Institute of Science and Technology in Japan) and Prof.\ Hiroyoshi Mitake (The University of Tokyo in Japan) for many helpful discussions. 
The work of the author is supported by JSPS Grants-in-Aid for Scientific Research, No.\ JP23K12992, JP23K20802 and JP23H00085. 

\section{Definition of viscosity solutions}\label{sec:viscosity} 

We need to adopt the viscosity solution theory \cite{CIL, Gbook} to handle \eqref{flow eq}, which is in general fully nonlinear and degenerate parabolic. 
For convenience of notation, hereafter we denote
\[ Q:= (-b, b) \times (0,\infty), \quad Q_0:= (-b, b) \times [0,\infty), \quad Q_0^T := (-b,b) \times [0,T) \]
for any $T>0$ and let
\begin{equation}\label{def F} 
F(p,z) := -f(g(p)z) 
\end{equation}
for $p, z \in \R$. 
We also denote by, respectively, $u^*$ and $u_*$ the upper and lower semicontinuous envelope on $\ol{Q}$ of a function $u$ defined on $Q$, that is, 
\begin{align}
&u^*(x,t) := \lim_{r \to +0} \sup\{u(y,s) : (y,s) \in Q_0 \cap B_r((x,t))\} \quad \text{for} \; \; (x,t) \in \ol{Q}, \label{envelope} \\
&u_*(x,t) := \lim_{r \to +0} \inf\{u(y,s) : (y,s) \in Q_0 \cap B_r((x,t))\} \quad \text{for} \; \; (x,t) \in \ol{Q}, \label{envelope2} 
\end{align}
where $B_r((x,t))$ is the ball with center $(x,t)$ and radius $r$. 

\begin{defi}\label{def singular}
An upper semicontinuous function $u$ in $Q_0$ is called a sub-solution of \eqref{flow eq} and \eqref{bdry} if whenever there exist $(x_0, t_0)\in Q$ and $\phi\in C^2(Q)$ such that $u-\phi$ attains a local maximum at $(x_0, t_0)$, we have 
\[
\phi_t(x_0, t_0)+F\left(\phi_x(x_0, t_0), \phi_{xx}(x_0, t_0)\right)\leq 0.
\]
A lower semicontinuous function $u$ in $Q_0$ is called a super-solution of \eqref{flow eq} and \eqref{bdry} if the following conditions hold:
\begin{enumerate}
\item[(i)] Whenever there exist $(x_0, t_0)\in Q$ and $\phi\in C^2(Q)$ such that $u-\phi$ attains a local minimum at $(x_0, t_0)$, we have 
\[
\phi_t(x_0, t_0)+F\left(\phi_x(x_0, t_0), \phi_{xx}(x_0, t_0)\right)\geq 0.
\]
\item[(ii)] For any $t>0$, it holds that $\lim_{x \to \pm b} u(x,t) = \infty$. 
\end{enumerate}
A function $u$ defined in $Q_0$ is said to be a solution of \eqref{flow eq}--\eqref{initial} if $u^*$ is a sub-solution, $u_*$ is a super-solution of \eqref{flow eq} and \eqref{bdry}, and it holds that $u^*(\cdot, 0) = u_*(\cdot, 0) = u_0$.
\end{defi}

\section{Comparison principle}\label{sec:comparison}

In this section we provide several comparison results. 
As we discuss the difficulty to establish the comparison theory in the introduction, we study the comparison results step by step. 
We first prove the comparison theory between unbounded super-solution and bounded sub-solution. 

\begin{thm}\label{thm:com1}
Let $b>0$. 
Assume that $f, g\in C(\R)$ satisfy (A1) and (A2). 
Let $u$ and $v$ be respectively a sub- and super-solution of \eqref{flow eq} and \eqref{bdry}. 
Assume $u(\cdot, t)$ is bounded from above at any time $t \ge 0$. 
If $u(\cdot, 0) \le v(\cdot, 0)$ in $(-b, b)$, then $u \le v$ in $Q_0$. 
\end{thm}

\begin{proof}
Assume by contradiction that $u(\bar{x}, \bar{t}) - v(\bar{x}, \bar{t}) > 0$ at some $(\bar{x}, \bar{t}) \in Q$. 
Then, there exists $T > 0$ large such that 
\begin{equation}\label{positive-com1} 
u(\bar{x},\bar{t}) - v(\bar{x},\bar{t}) - \frac{1}{T-\bar{t}} > 0. 
\end{equation}
Let $u^*$ be the upper semicontinuous envelope of $u$ defined by \eqref{envelope}. 
Notice that $u^* = u$ in $Q_0$ and $u^*(\cdot, t)$ is bounded from above due to the upper semicontinuity and the boundedness assumption of $u$. 
For any $\varepsilon > 0$, let 
\[ \Phi_\varepsilon(x,y,t) := u^*(x,t) - v(y,t) - \frac{1}{T-t} - \frac{|x-y|^2}{\varepsilon}\]
for $(x,y,t) \in [-b,b] \times (-b,b) \times [0,T)$. 
From \eqref{positive-com1} and $\lim_{x \to \pm b} v(x,t) = \infty$, for any $\varepsilon > 0$, there exists $(x_\varepsilon, y_\varepsilon, t_\varepsilon) \in [-b,b] \times (-b,b) \times [0,T)$ such that 
\begin{equation}\label{positive-Phi1} 
\sup_{[-b,b] \times (-b,b) \times [0,T)} \Phi_\varepsilon(x,y,t) = \Phi_\varepsilon(x_\varepsilon, y_\varepsilon, t_\varepsilon) \ge u(\bar{x},\bar{t}) - v(\bar{x},\bar{t}) - \frac{1}{T-\bar{t}} > 0. 
\end{equation}
The relation \eqref{positive-Phi1} also enables us to deduce that 
\[ \frac{|x_\varepsilon - y_\varepsilon|^2}{\varepsilon} \le u^*(x_\varepsilon, t_\varepsilon) - u(\bar{x}, \bar{t}) + v(\bar{x}, \bar{t}) - v(y_\varepsilon, t_\varepsilon) -\frac{1}{T-\bar{t}}, \]
which implies the existence of $(x_0, t_0) \in \overline{Q}$ such that along a subsequence 
\[ x_\varepsilon, y_\varepsilon \to x_0, \quad t_\varepsilon \to t_0\]
as $\varepsilon \to 0$. 
Hereafter we still index the converging subsequence by $\varepsilon$ for convenience of notion. 
Using \eqref{positive-Phi1} and $\lim_{x \to \infty} v(x,t) = \infty$ again, we obtain $x_0 \neq \pm b$. 
Due to $u(\cdot, 0) \le v(\cdot, 0)$, we have also $t_0 > 0$. 
Therefore, $(x_\varepsilon, y_\varepsilon, t_\varepsilon) \in (-b, b) \times (-b, b) \times (0,T)$ for $\varepsilon > 0$ small and we fix such $\varepsilon > 0$. 

We next apply the Crandall-Ishii lemma (cf.\ \cite[Theorem 8.3]{CIL}) to find $(\tau_1, p_1, z_1) \in \ol{P}^{2, +} u(x_\varepsilon, t_\varepsilon)$ and $(\tau_2, p_2, z_2) \in \ol{P}^{2,-} v(y_\varepsilon, t_\varepsilon)$ for any $\sigma > 0$ small such that 
\begin{equation}\label{com-derivatives1} 
\tau_1 - \tau_2 = \frac{1}{(T - t_\varepsilon)^2} \ge \frac{1}{T^2}, \quad p_1 = p_2 = \frac{2(x_\varepsilon - y_\varepsilon)}{\varepsilon}, 
\end{equation}
and 
\begin{equation}\label{com-second1} 
\begin{pmatrix}
z_1 & 0\\
0 & - z_2
\end{pmatrix}
\leq
\dfrac{2}{\vep}\begin{pmatrix} 1 & -1 \\ -1 & 1 \end{pmatrix}+{4\sigma\over \vep^2}\begin{pmatrix} 1 & -1 \\ -1 & 1 \end{pmatrix}^2.
\end{equation}
Multiplying both sides of the last inequality \eqref{com-second1} by the vector $(1,1)$ from left and from right, we deduce that $z_1 \le z_2$. 
We then adopt the definition of viscosity sub- and super-solutions for $u$ and $v$ respectively to deduce that 
\[ \tau_1 + F(p_1, z_1) \le 0, \quad \tau_2 + F(p_2, z_2) \ge 0, \]
which yields by the assumption (A1) and \eqref{com-derivatives1}
\[ 0 < f^{-1}(\tau_1) - f^{-1}(\tau_2) \le g(p_1) z_1 - g(p_2) z_2 \le 0. \]
Therefore, we have a contradiction. 
\end{proof}

We next provide the comparison theory when a sub-solution is not bounded. 
We compare between super- or sub-solution to \eqref{scale eq1} or \eqref{scale eq2} with \eqref{bdry} and sub- or super-solution to \eqref{flow eq} and \eqref{bdry}, respectively. 

\begin{thm}\label{thm:com-conti}
Let $b>0$. 
Assume that $f, g \in C(\mathbb{R})$ satisfy (A1) and (A2). 
Let $u$ and $v$ be respectively sub- and super-solution of \eqref{flow eq} and \eqref{bdry}. 
Assume $u$ satisfies $\lim_{x \to \pm b} u(x, t) = \infty$ for $t \ge 0$. 
Assume also either of the following conditions: 
\begin{enumerate}
\item[(a)] $v$ is continuous on $Q_0$; there exists $\delta > 0$ such that $v$ satisfies 
\[ v_t \ge \max\{f((1+\delta)g(v_x) v_{xx}), 0\} \quad \text{for} \; \; (x,t) \in (-b, b) \times (0, \infty) \]
in the viscosity sense; and there exists $b_0 \in (0, b)$ such that $v_{xx}(\cdot, 0) \ge 0$ in $(-b, -b_0) \cup (b_0, b)$ in the viscosity sense. 
\item[(b)] $u$ is continuous on $Q_0$; $u$ satisfies 
\[ u_t \le \min\{ f(g(u_x) u_{xx}), 0\} \quad \text{for} \; \; (x,t) \in (-b, b) \times (0, \infty) \]
in the viscosity sense; and there exists $b_0 \in (0, b)$ such that $u_{xx}(\cdot, 0) \ge 0$ in $(-b, -b_0) \cup (b_0, b)$ in the viscosity sense. 
\end{enumerate}
If $u(\cdot, 0) \le v(\cdot, 0)$ in $(-b, b)$, then $u \le v$ in $Q_0$. 
\end{thm}

\begin{proof}
We first prove the theorem in the case that (a) holds. 
For $\lambda > 0$, let 
\[ v_\lambda (x,t) := \frac{1}{1+\lambda} v((1+\lambda) x, (1+\lambda) t) \quad \text{for} \; \; -\frac{b}{1+\lambda} < x < \frac{b}{1+\lambda}, \; \; t \ge 0. \]
Then, for $0 < \lambda < \delta$, we have by a simple calculation 
\begin{align*} 
(v_\lambda)_t \ge&\; \max\left\{f\left(\frac{1+\delta}{1+\lambda} g((v_\lambda)_x) (v_\lambda)_{xx}\right), 0\right\} \\
\ge&\; f(g((v_\lambda)_x) (v_\lambda)_{xx}) \quad \text{in} \; \; \left(-\frac{b}{1+\lambda}, \frac{b}{1+\lambda}\right) \times (0, \infty) 
\end{align*}
in the viscosity sense. 
Furthermore, it holds that 
\begin{equation}\label{order-initial-time} 
v(x, 0) \le v_{\lambda}(x, 0) + \omega(\lambda) \quad \text{for} \; \; -\frac{b}{1+\lambda} < x < \frac{b}{1+\lambda}, 
\end{equation}
where $\omega$ is a modulus of continuity. 
Indeed, since $v(\cdot, 0)$ is continuous, it holds that 
\[ |v(x,0) - v_{\lambda}(x,0)| \le \omega(\lambda) \quad \text{for} \; \; -b_0 \le x \le b_0 \]
if $\lambda$ is sufficiently small. 
Due to the convexity assumption in (a), we can see that $v(\cdot, 0)$ is differentiable in $(b_0, b)$ and $v_x(\cdot, 0)$ is increasing in $(b_0, b)$. 
Therefore, for $x \in (b_0, \frac{b}{1+\lambda})$, we have 
\begin{align*} 
v_\lambda(x,0) =&\; \int_{b_0}^x v_x((1+\lambda) \tilde{x}, 0) \; d\tilde{x} + v_\lambda(b_0, 0) \\
\ge&\; \int_{b_0}^x v_x(\tilde{x}, 0) \; d\tilde{x} + v(b_0, 0) - \omega(\lambda) = v(x, 0) - \omega(\lambda). 
\end{align*}
Similarly, we have $v_\lambda(x, 0) + \omega(\lambda) \ge v(x, 0)$ for $x \in (-\frac{b}{1+\lambda}, -b_0)$. 
We thus obtain \eqref{order-initial-time}. 

Since $u(\cdot, t)$ is bounded from above in $(-\frac{b}{1+\lambda}, \frac{b}{1+\lambda})$ for $t \ge 0$, the comparison result in Theorem \ref{thm:com1} yields 
\[ u(x, t) \le v_\lambda(x,t) + \omega(\lambda) \quad \text{for} \; \; (x,t) \in \left(-\frac{b}{1+\lambda}, \frac{b}{1+\lambda}\right) \times [0, \infty) \]
if $u(\cdot, 0) \le v(\cdot, 0)$ in $(-b, b)$. 
Letting $\lambda \to 0$, we have $u \le v$ in $Q_0$ due to the continuity of $v$. 

When (b) holds, we can prove the comparison result similarly by using the scaling $u_\lambda(x,t) := (1+\lambda) u(\frac{x}{1+\lambda}, \frac{t}{1+\lambda})$. 
\end{proof}

A similar scaling argument as in the proof of Theorem \ref{thm:com-conti} works well if $u(\cdot, t)$ is convex for any $t \ge 0$ and $u$ satisfies a condition stronger than the inequality in (b) of Theorem \ref{thm:com-conti}. 
The following comparison result will be used when we prove (a) in Theorem \ref{thm:non-existence}. 

\begin{thm}\label{thm:com-conti2}
Let $b>0$. 
Assume that $f, g \in C(\mathbb{R})$ satisfy (A1) and (A2). 
Let $u$ and $v$ be respectively sub- and super-solution of \eqref{flow eq} and \eqref{bdry}. 
Assume $u(\cdot, t)$ is convex for any $t \ge 0$ and $u$ is a continuous function satisfying  
\[ u_t \le f\left(\frac{1}{1+\delta} g(u_x) u_{xx} \right) \quad \text{for} \; \; (x,t) \in (-b, b) \times (0, \infty) \]
for some $\delta > 0$. 
If $u(\cdot, 0) \le v(\cdot, 0)$ in $(-b, b)$, then $u \le v$ in $Q_0$. 
\end{thm}

\begin{proof}
Letting 
\[ u_\lambda (x,t) := (1+\lambda) u\left(\frac{x}{1+\lambda}, \frac{t}{1+\lambda}\right) \quad \text{for} \; \; -b < x < b, \]
we can see that (1) $u_\lambda$ is a sub-solution to \eqref{flow eq} for $0 < \lambda \le \delta$; (2) $u_\lambda(\cdot, t)$ is bounded on $(-b, b)$ for any $t \ge 0$ and $\lambda > 0$; and (3) there is a modulus of continuity $\omega$ such that $u_\lambda(\cdot, 0) \le u(\cdot, 0) + \omega(\lambda)$ whether $u(\cdot, 0)$ is bounded on $(-b, b)$ or $\lim_{x \to \pm b} u(x, 0)=\infty$. 
Therefore, the claim can be proved by a similar argument as in the proof of Theorem \ref{thm:com-conti}. 
\end{proof}

We finally provide the comparison theory between general sub- and super-solution to \eqref{flow eq}--\eqref{bdry} under an assumption corresponding to (B3) with suitable constants $\gamma_\pm$. 
Since Theorem \ref{thm:com-conti} can not be applied to sub- and super-solutions to \eqref{flow eq}--\eqref{bdry} in general, we have to construct another sub- or super-solution to estimate the divergence rate of above sub- and super-solution at the boundary points $x = \pm b$ for $t > 0$. 
Therefore, we introduce functions $h_{\gamma_\pm, D_\pm} \in C^2((-b,b))$ for $\gamma_+, \gamma_- \ge 0$ and $D_+, D_- >0$ satisfying the following condition to construct another sub- and super-solution satisfying the conditions in (a) or (b) of Theorem \ref{thm:com-conti} to estimate the divergence rate: 
\begin{itemize}
\item[(h1)] There exists $b_0 \in (0,b)$ such that 
\[ h_{\gamma_\pm, D_\pm}(x) = \begin{cases} 
D_+\psi_{\gamma_+}(b-x)  & \text{if} \; \; x \in [b_0, b), \\
D_-\psi_{\gamma_-}(b+x)  & \text{if} \; \; x \in (-b, -b_0], 
\end{cases}\] 
where $\psi_{\gamma_\pm}$ is the function defined in \eqref{def-psi}
\end{itemize}

We state the comparison result, which yields the continuity and the uniqueness results in Theorem \ref{thm:existence} and Theorem \ref{thm:existence2}, as follows: 

\begin{thm}\label{thm:com2}
Let $b>0$. 
Assume that $f$ and $g$ satisfy (A1) and (A2) with $\alpha >1$. 
Let $u$ and $v$ be respectively a sub- and a super-solution of \eqref{flow eq} and \eqref{bdry}. 
Assume that there exists constants $\gamma_\pm^u, \gamma_\pm^v \ge \max\{\frac{2-\alpha}{\alpha-1}, 0\}$ and $D_\pm^u, D_\pm^v > 0$ such that $u(\cdot, 0)$ and $v(\cdot,0)$ respectively satisfy (B3) replaced $u_0, \gamma_\pm, D_\pm$ by $u(\cdot, 0), \gamma_\pm^u, D_\pm^u$ and $v(\cdot, 0), \gamma_\pm^v, D_\pm^v$. 
If $u(\cdot, 0) \le v(\cdot, 0)$, then $u \le v$ in $Q_0$. 
\end{thm}

Before beginning the proof of the comparison principle, we estimate the divergence rate of $u(\cdot, t)$ and $v(\cdot, t)$ at the boundary points $x = \pm b$. 

\begin{lem}\label{lem:initial-conti}
Let $b>0$. 
Assume that $f$ and $g$ satisfy (A1) and (A2) with $\alpha >1$.  
Let $u$ and $v$ be respectively a sub- and a super-solution of \eqref{flow eq} and \eqref{bdry}. 
Assume that there exists constants $\gamma_\pm^u, \gamma_\pm^v \ge \max\{\frac{2-\alpha}{\alpha-1}, 0\}$ and $D_\pm^u, D_\pm^v > 0$ such that $u(\cdot, 0)$ and $v(\cdot,0)$ respectively satisfy (B3) replaced $u_0, \gamma_\pm, D_\pm$ by $u(\cdot, 0), \gamma_\pm^u, D_\pm^u$ and $v(\cdot, 0), \gamma_\pm^v, D_\pm^v$. 
Let $h_{\gamma_\pm^u, D_\pm^u}$ and $h_{\gamma_\pm^v, D_\pm^v}$ be functions satisfying (h1). 
Then, there exists $\tilde{C}^u, \tilde{C}^v \ge 0$ such that 
\begin{equation}\label{pre-semi-order} 
u(x,t) \le h_{\gamma_\pm^u, D_\pm^u}(x) + \tilde{C}^u (1+t), \quad v(x,t) \ge h_{\gamma_\pm^v, D_\pm^v}(x) - \tilde{C}^v(1+t) \quad \text{for} \; \; (x,t) \in Q_0. 
\end{equation}
Furthermore, it holds that 
\begin{equation}\label{u-ini-conti} 
\limsup_{(x,t) \to (\pm b, 0)} \left(u(x,t) - h_{\gamma_\pm^u, D_\pm^u}(x)\right)= \lim_{x \to \pm b} \left(u(x,0) - h_{\gamma_\pm^u, D_\pm^u}(x)\right) 
\end{equation}
and 
\begin{equation}\label{v-ini-conti} 
\liminf_{(x,t) \to (\pm b, 0)} \left(v(x,t) - h_{\gamma_\pm^v, D_\pm^v}(x)\right) = \lim_{x \to \pm b} \left(v(x,0) - h_{\gamma_\pm^v, D_\pm^v}(x) \right). 
\end{equation}
\end{lem}

\begin{rmk}\label{rk:preservation-order}
If $u$ is a viscosity solution satisfying (B3) with $\gamma_\pm \ge \max\{\frac{2-\alpha}{\alpha - 1}, 0\}$, the property \eqref{pre-semi-order} replaced respectively $u, v$ by $u^*, u_*$ holds. 
The estimate \eqref{pre-semi-order} thus also shows that, letting $h_{\gamma_\pm, D_\pm}$ be a function satisfying (h1), for any $T>0$, there exists $\hat{C}_T >0$ such that 
\[ |u(x,t) - h_{\gamma_\pm, D_\pm}(x)| \le \hat{C}_T \quad \text{for} \; \; (x,t) \in Q_0^T, \]
for each $t \ge 0$. 
Therefore, the function $u(\cdot, t)$ preserve the divergence rate assumed in (B3). 
\end{rmk}

\begin{proof}[Proof of Lemma \ref{lem:initial-conti}]
We prove the first inequality in \eqref{pre-semi-order} and \eqref{u-ini-conti} since the other estimate can be proved similarly. 
Due to the assumption (B3) for $u$, we can choose a continuous function $U$ defined in $\mathbb{R}$ satisfying $U \equiv u(\cdot, 0) - h_{\gamma_{\pm}^u, D_{\pm}^u}$ in $(-b, b)$. 
We not that it holds that 
\begin{equation}\label{conti-bdry-U}
U(\pm b) = \lim_{x \to \pm b} u(x, 0) - h_{\gamma_{\pm}^u, D_{\pm}^u}(x). 
\end{equation}
Let $\rho_\varepsilon$ be a mollifier and $u_{\varepsilon}(x) := \rho_\varepsilon * U(x) + h_{\gamma_{\pm}^u, D_{\pm}^u}(x)$ for $x \in [-b, b]$. 
We then obtain  
\begin{equation}\label{modulus}
\sup_{x \in (-b, b)} |u_\varepsilon(x) - u(x, 0)| = \omega(\varepsilon), 
\end{equation}
where $\omega$ is a modulus of continuity. 
Furthermore, due to $(h_{\gamma_{\pm}^u, D_{\pm}^u})_{xx} (x) \to \infty$ as $x \to \pm b$ and $\rho_\varepsilon * U \in C^\infty([-b, b])$, for any $\varepsilon > 0$ there exists $b_0 \in (0, b)$ such that 
\begin{equation}\label{convex-u-e}
(u_\varepsilon)_{xx}(x) \ge 0 \quad \text{for}  \; \; x \in (-b, -b_0) \cup (b_0, b).
\end{equation} 

We hereafter fix arbitrary $\varepsilon > 0$ and construct a sub-solution. 
From the assumption (A2), there exist $L>0$ and $\hat{C} > 0$ such that 
\[ g(s) \le \hat{C} |s|^{-\alpha} \quad \text{for} \; \; |s| > L. \]
Therefore, the condition (h1) and $\rho_\varepsilon * U \in C^\infty ([-b, b])$ yield that there exists $b_1 \in (0, b)$ such that 
\[ |(u_\varepsilon)_x(x)| > L \quad \text{for} \; \; |x| \in [b_1, b), \]
and thus, letting $\hat{\alpha} := \max\{\frac{2-\alpha}{\alpha-1}, 0\}$, it holds that 
\[ g((u_\varepsilon)_x(x)) (u_\varepsilon)_{xx}(x) \le C(\min\{b-x, b+x\})^{\hat{\alpha}(\alpha-1) + \alpha -2} \quad \text{for} \; \; |x| \in [b_1, b) \]
for some $C>0$. 
Since $\hat{\alpha}(\alpha - 1) + \alpha -2 > 0$, if we fix $\delta > 0$, there exists $M > 0$ such that 
\[ f\left((1+\delta) g((u_\varepsilon)_x(x)) (u_\varepsilon)_{xx}(x) \right) \le M \quad \text{for} \; \; x \in (-b, b). \]
Therefore, letting 
\[ \tilde{v}(x,t) := u_\varepsilon(x) + \omega(\varepsilon) + M t \quad \text{for} \; \; (x,t) \in [-b, b] \times [0, \infty) , \] 
since $\tilde{v}$ satisfies the condition (a) in Theorem \ref{thm:com-conti} due to \eqref{modulus} and \eqref{convex-u-e}, we can apply the comparison result to obtain $u \le \tilde{v}$ on $Q_0$. 
It yields the estimate \eqref{pre-semi-order} for $u$ with 
\[ \tilde{C}^u := \max\left\{\left(\sup_{x \in [-b, b]} |\rho_\varepsilon * U (x)|\right) + \omega(\varepsilon), M\right\}. \] 

Taking the limit superior of $u - h_{\gamma_{\pm}^u, D_{\pm}^u} \le \tilde{v} - h_{\gamma_{\pm}^u, D_{\pm}^u}$ as $(x,t) \to (\pm b, 0)$ and letting $\varepsilon \to 0$, we have by \eqref{conti-bdry-U}
\[ \limsup_{(x,t) \to (\pm b, 0)} u(x,t) - h_{\gamma_{\pm}^u, D_{\pm}^u}(x) \le U(\pm b) = \lim_{x \to \pm b} u(x, 0) - h_{\gamma_{\pm}^u, D_{\pm}^u}(x). \]
Since it obviously holds that 
\[ \limsup_{(x,t) \to (\pm b, 0)} u(x,t) - h_{\gamma_{\pm}^u, D_{\pm}^u}(x) \ge \lim_{x \to \pm b} u(x, 0) - h_{\gamma_{\pm}^u, D_{\pm}^u}(x), \]
we have \eqref{u-ini-conti}. 
\end{proof}

Let us now prove Theorem \ref{thm:com2}. 

\begin{proof}[Proof of Theorem \ref{thm:com2}]
Let $h_u \in C^2((-b,b))$ be a convex function satisfying (h1) with $\gamma_\pm = \gamma_\pm^u$ and $D_\pm = D_\pm^u$. 
Due the the assumption (B3) for $u$, there exists $m \in \mathbb{R}$ such that 
\[ u(x, 0) \ge h_u(x) - m \quad \text{for} \; \; x \in (-b,b). \]
Since $h_u$ is a convex function, $h_u$ is a stationary sub-solution to \eqref{flow eq} and \eqref{bdry}. 
By a standard argument, we can see that 
\[ \tilde{u}(x,t) := \max\{u(x,t), h_u(x)-m\} \quad \text{for} \; \; (x,t) \in Q_0 \]
is also a sub-solution to \eqref{flow eq} and \eqref{bdry} satisfying $\tilde{u}(\cdot, 0) \equiv u(\cdot, 0)$. 
We will prove $\tilde{u} \le v$ on $Q_0$, which yields $u \le v$ on $Q_0$ due to the definition of $\tilde{u}$. 

We note that the definition of $\tilde{u}$ and \eqref{pre-semi-order} yield that, for any $T>0$, there exists $C_T > 0$ such that 
\begin{equation}\label{bdd-u-h}
|\tilde{u}(x,t) - h_u(x)| \le C_T \quad \text{for} \; \; (x,t) \in Q_0^T. 
\end{equation}
Our proof consists of several steps. 

\ul{\textit{Step 1. Sup-convolution}}

We first take the sup-convolution for $\tilde{u}$ with respect to the time variable. 
Fix $T > 0$ arbitrarily.  
For any $\delta>0$ and $(x, t)\in Q_0^T$, let 
\[
u_\delta(x, t)=\sup_{s\in [0, T+1]} \left\{\tilde{u}(x, s)-{|t-s|^2\over \delta}\right\}.
\]
By the definition, it is clear that 
\begin{equation}\label{remove_lipa}
\tilde{u} \leq u_{\delta_1}\leq u_{\delta_2} \quad \text{in} \; \;  Q_0^T \quad \text{for all} \; \; 0<\delta_1\leq \delta_2. 
\end{equation} 
Due to the upper semicontinuity of $\tilde{u}$, there exists $s_\delta(x, t)\in [0, T+1]$ such that 
\beq\label{sup-convolution max}
u_\delta(x, t)=\tilde{u}(x, s_\delta(x, t))-{|t-s_\delta(x, t)|^2\over \delta}.
\eeq
From $u_\delta(x, t)\geq \tilde{u}(x, t)$ and \eqref{bdd-u-h}, we have 
\begin{equation}\label{remove_lipb}
|t - s_\delta(x, t)| \le  \{(|\tilde{u}(x, s_\delta(x,t)) - h_u(x)| \delta + |\tilde{u}(x,t) - h_u(x)|) \delta \}^{1/2} \le \sqrt{2 C_{T+1} \delta}. 
\end{equation}
Hereafter we take $\delta>0$ satisfying 
\beq\label{delta range}
\sqrt{2 C_{T+1} \delta}<\min\{1,\ T\}.
\eeq
It is also not difficult to see that $u_\delta$ is upper semicontinuous and satisfies 
\begin{equation}\label{u-delta-bdd}
|u_\delta(x,t) - h_u(x,t)| \le C_{T+1} \quad \text{for} \; \; (x,t) \in Q_0^T
\end{equation} 

Due to \eqref{bdd-u-h} and \eqref{u-delta-bdd}, we can expand the functions $\tilde{u}-h_u$ and $u_\delta - h_u$ to be defined on $\overline{Q_0^T}$ by taking the upper semicontinuous envelopes $(\tilde{u} - h_u)^*$ and $(u_\delta - h_u)^*$ defined as in \eqref{envelope}, respectively. 
Notice the continuity of $\tilde{u}(\cdot, 0)$ on $(-b, b)$ and the estimate \eqref{u-ini-conti} yield the continuity of $(\tilde{u} - h_u)^*(\cdot, 0)$ on $[-b, b]$. 
Then, we can get 
\beq\label{remove lip4}
\sup_{x \in [-b, b]} (u_\delta - h_u)^* (x,0) - (\tilde{u}-h_u)^* (x,0) \leq \omega(\delta), 
\eeq
where $\omega$ is a modulus of continuity. 
Indeed, since $(u_\delta - h_u)^*(\cdot, 0)$ is upper semicontinuous and $(\tilde{u} - h_u)^*(\cdot, 0)$ is continuous on $[-b, b]$, the function $(u_\delta - h_u)^* (\cdot,0) - (\tilde{u}-h_u)^* (\cdot,0)$ attains a maximum at some point $x_\delta \in [-b, b]$. 
We thus obtain 
\[ \sup_{x \in [-b, b]} (u_\delta - h_u)^* (x,0) - (\tilde{u} - h_u)^* (x,0)  = (\tilde{u}-h_u)^*(x_\delta, s_\delta(x_\delta,0)) - (\tilde{u}-h_u)^*(x_\delta, 0). \] 
Therefore, since $s_\delta(x_\delta, 0) \in [0, \sqrt{2C_{T+1} \delta}]$ follows from \eqref{remove_lipb}, the upper semicontinuity of $(\tilde{u}-h_u)^*$ on $\overline{Q_0^T}$ and the continuity of $(\tilde{u}-h_u)^*(\cdot, 0)$ on $[-b, b]$ yield 
\[ \limsup_{\delta \to +0} \sup_{x \in [-b, b]} (u_\delta - h_u)^* (x,0) - (\tilde{u}-h_u)^*(x,0) \le 0. \]
This implies \eqref{remove lip4}. 

Moreover, we can apply a standard argument to show that $u_\delta$ is a subsolution in $Q_{T, \delta}:=~(-b, b)\times (\sqrt{2C_{T+1} \delta}, T)$. 
Indeed, if there exists $\phi\in C^2(\ol{Q})$ such that $u_\delta-\phi$ attains a local maximum at some $(x_0, t_0)\in Q_{T, \delta}$ with $\phi(x_0, t_0)=u_\delta(x_0, t_0)$, then 
\beq\label{remove lip2}
(x, t, s) \mapsto \tilde{u}(x, s)-{|t-s|^2\over \delta}-\phi(x, t)
\eeq
attains a local maximum in $(-b,b) \times (\sqrt{2 C_{T+1} \delta}, T) \times (0,T+1)$ at $(x_0, t_0, s_\delta(x_0, t_0))$. 
Since $t_0 > \sqrt{2C_{T+1} \delta}$, it follows from \eqref{remove_lipb} and \eqref{delta range} that $0<s_\delta(x_0, t_0) < T+1$. 
Noticing that
\[
(x, s)\mapsto \tilde{u}(x, s)- {|t_0-s|^2\over \delta}-\phi(x, t_0)
\]
attains a maximum at $(x_0, s_\delta(x_0, t_0))$, we may apply the definition of subsolutions to get 
\beq\label{remove lip3}
{2(s_\delta(x_0, t_0)-t_0)\over \delta}+F(\phi_x(x_0, t_0), \phi_{xx}(x_0, t_0))\leq 0.
\eeq
On the other hand, the maximality in \eqref{remove lip2} also implies that 
\[
{2(s_\delta(x_0, t_0)-t_0)\over \delta}=\phi_t(x_0, t_0),
\]
which, combined with \eqref{remove lip3}, completes the verification of the subsolution property of $u_\delta$ at every $(x_0, t_0)\in Q_{T, \delta}$.

Another important property of $u_\delta$ is its Lipschitz continuity in time; for any fixed $x\in (-b, b)$, 
\beq\label{time lip}
|u_\delta(x, t_1)-u_\delta(x, t_2)|\leq L_\delta |t_1-t_2| \quad \text{for} \; \; t_1, t_2 \in [0,T], \; \; x \in (-b,b),  
\eeq
where $L_\delta>0$ is a constant depending only on $T > 0$ and $\delta>0$. 
Indeed, by direct calculation we obtain 
\[\begin{aligned}
&\; u_\delta(x, t_1) - u_\delta(x,t_2) \\
\le&\; \left\{\tilde{u}(x, s_\delta(x,t_1)) - {|t_1-s_\delta(x,t_1)|^2\over \delta}\right\} - \left\{\tilde{u}(x, s_\delta(x,t_1)) - {|t_2-s_\delta(x,t_1)|^2\over \delta}\right\} \\
=&\; \frac{1}{\delta}(t_1+t_2 - s_\delta(x,t_1)) (t_2 - t_1) \le \frac{3T +1}{\delta} |t_1 - t_2|. 
\end{aligned}\]
Interchanging the roles of $t_1$ and $t_2$ enables us to obtain a symmetric estimate. We thus get \eqref{time lip} with
$L_\delta=(3T+1)/\delta.$

\ul{\textit{Step 2. Choice of $\delta$ for comparison near t=0}}

Let us now proceed to the main part of the proof. 
Assume by contradiction that  there exists $(x_0, t_0) \in Q$ such that $\tilde{u}(x_0, t_0) - v(x_0, t_0) > 0$. 
Then, there exists $T>t_0$ large and $\mu > 0$ small such that 
\[ \tilde{u}(x_0,t_0) - v(x_0,t_0) - \frac{1}{T-t_0} > \mu. \]
Fix such $T>0$ and let $u_\delta$ be the sup-convolution of $\tilde{u}$ with this $T$, where $\delta>0$ is taken to satisfy \eqref{delta range}. 

It follows from the monotonicity \eqref{remove_lipa} that 
\begin{equation} \label{contra cond} 
\Psi_\delta (x_0, t_0) > \mu 
\end{equation}
for any $\delta > 0$, where 
\[
\Psi_\delta(x, t)=u_\delta(x, t)-v(x, t)-{1 \over T-t} \quad \text{for} \; \; (x,t) \in Q_0^T. 
\]
We note that, from $u(\cdot, 0) \le v(\cdot, 0)$, the constants $D_\pm^v$ and $\gamma_\pm^v$ related to the divergence rate of $v_0$ around the boundary points satisfy $\gamma_E^v \ge \gamma_E^u$ respectively for $E=+$ and $E=-$, and, in particular, $D_E^v \ge D_E^u$ if $\gamma_E^v = \gamma_E^u$. 
Therefore, letting $(h_u - v)_*$ be the lower semicontinuous envelope as in \eqref{envelope2}, we have $(v-h_u)_*(\pm b, t) \in (-\infty, \infty]$ for $t \ge 0$ due to \eqref{pre-semi-order}. 
Hereafter, we extend the domain of $\Psi$ as 
\[ \hat{\Psi}_\delta(x,t) := (u_\delta - h_u)^*(x,t) - (v - h_u)_* (x,t) - \frac{1}{T-t} \quad \text{for} \; \; (x,t) \in [-b, b] \times [0,T). \]
Notice that $\hat{\Psi}_\delta = \Psi_\delta$ holds on $Q_0^T$.

We next prove that there exists $\delta$ small such that 
\begin{equation}\label{remove_lipd} 
\sup_{[-b, b] \times [0, 2 \sqrt{2 C_{T+1} \delta}]} \hat{\Psi}_\delta \le -\frac{1}{4T}. 
\end{equation}
Indeed, 
In view of \eqref{remove lip4} and $\tilde{u}(\cdot, 0) = u(\cdot, 0) \le v(\cdot, 0)$, we can take $\delta_1 > 0$ sufficiently small to deduce that 
\begin{align*}
\sup_{x\in [-b, b]} \hat{\Psi}_{\delta_1}(x, 0) \le &\; \sup_{x\in [-b, b]} \{(\tilde{u} - h_u)^*(x, 0)-(v - h_u)_*(x, 0)\} + \omega(\delta_1) -{1 \over T} \\
=&\; \sup_{x \in (-b,b)} \{\tilde{u}(x, 0) - v(x, 0)\} + \omega(\delta_1) -{1 \over T} \le -{1 \over 2T}.
\end{align*}
Here, the properties \eqref{u-ini-conti} and \eqref{v-ini-conti} have been used in the equation transformation if $(\gamma_+^v, D_+^v) = (\gamma_+^u, D_+^u)$ or $(\gamma_-^v, D_-^v) = (\gamma_-^u, D_-^u)$. 
Therefore, we have by the upper semicontinuity of $(u_{\delta_1} - h_u)^* - (v - h_u)_*$ 
\begin{align*}
&\; \limsup_{t \to 0} \sup_{x \in[-b, b]} \hat{\Psi}_{\delta_1}(x,t) \le \sup_{x \in [-b, b]} \hat{\Psi}_{\delta_1}(x,0) \le -\frac{1}{2T}, 
\end{align*}
which yields that there exists $t_1 > 0$ such that 
\[ \sup_{ [-b,b] \times [0, t_1]} \hat{\Psi}_{\delta_1} \le - \frac{1}{4T}. \]
Taking \eqref{delta range} into consideration, we choose $\delta>0$ such that 
\beq\label{delta range2}
\delta<\delta_1, \quad \sqrt{2 C_{T+1} \delta}<\min\left\{1,\ T,\ {t_1\over 2}\right\}. 
\eeq
We thus obtain \eqref{remove_lipd} by \eqref{remove_lipa}. 
Let us fix $\delta >0$ as in \eqref{delta range2}. 

\ul{\textit{Step 3. Scaling argument and comparison}}

For $\lambda \in (0,1)$, we define $\Psi_{\delta, \lambda}$ by 
\[ \Psi_{\delta, \lambda}(x,t) := u_\delta (\lambda (x-x_0) + x_0, t) - \lambda v(x, t) - \frac{1}{T-t} \quad \text{for} \; \; (x, t) \in Q_0^T. \]
Then, in view of \eqref{contra cond}, we have 
\begin{equation}\label{range lambda0} 
\Psi_{\delta, \lambda}(x_0, t_0) \ge \frac{\mu}{2} 
\end{equation}
for $\lambda \in (0,1)$ close to $1$. 

We take $\lambda\in (0, 1)$ sufficiently close to $1$ to also satisfy
\beq\label{range lambda1}
\left({1\over \lambda} - 1\right)L_\delta \leq {1 \over 2T^2}
\eeq
and 
\beq\label{range lambda2}
\min_{|s| \le L_\delta} \left\{\lambda f^{-1}(s) - f^{-1}\left(s - \dfrac{1}{2T^2}\right)\right\}>0. 
\eeq
Recall that $L_\delta>0$ appeared in \eqref{time lip}. 

Furthermore, since there exists a modulus of continuity $\omega$ such that 
\[ \sup_{x \in (-b, b)} h_u(\lambda (x-x_0) + x_0) - \lambda h_u(x) = \omega(1-\lambda), \]
we have 
\begin{align*} 
&\; \sup_{(-b, b)\times [0, 2 \sqrt{2C_{T+1} \delta}]} \Psi_{\delta, \lambda} \le \\
&\; \left(\sup_{[-b, b] \times [0, 2 \sqrt{2C_{T+1} \delta}]} (u_\delta - h_u)^*(\lambda (x-x_0) + x_0, t) - \lambda (v-h_u)_* (x, t) - \frac{1}{T-1} \right) + \omega(1-\lambda), 
\end{align*} 
which yields by \eqref{remove_lipd} and the semicontinuity of $(u_\delta - h_u)^*$ and $(v-h_u)_*$ 
\[ \limsup_{\lambda \to 1-0} \sup_{(-b, b)\times [0, 2 \sqrt{2C_{T+1} \delta}]} \Psi_{\delta, \lambda} \le \sup_{[-b, b] \times [0, 2 \sqrt{2C_{T+1} \delta}]} \hat{\Psi}_{\delta} \le -\frac{1}{4T}. \]
Therefore, it holds that 
\begin{equation}\label{range lambda3}
\sup_{(-b, b) \times [0, 2 \sqrt{2C_{T+1} \delta}]} \Psi_{\delta, \lambda} \le 0 
\end{equation}
when $\lambda$ is chosen to be close to $1$. 
We fix $\lambda\in (0, 1)$ that satisfies \eqref{range lambda0}--\eqref{range lambda3}. 

Let us now double the space variables. Consider, for any $\vep>0$ small, 
\[
\Phi_{\vep}(x, y, t):=u_\delta(x, t)-\lambda v(y, t)-{|(x-x_0)- \lambda (y-x_0)|^2\over \vep}-{1 \over T-t}
\]
for $x\in [-b, b]$, $y \in (-b,b)$ and $t\in [0, T)$.
Since 
\beq\label{remove lip5}
\Phi_{\vep}(x_0, x_0, t_0)=\Psi_{\delta, \lambda}(x_0, t_0) \ge \frac{\mu}{2}
\eeq
and $\lim_{y \to \pm b} v(y,t) = \infty$, $\Phi_{\vep}$ attains a positive maximum at some $(x_\vep, y_\vep, t_\vep)\in [-b, b]\times (-b, b) \times [0, T)$. 
The relation \eqref{remove lip5} also enables us to deduce that 
\[
{|(x_\vep - x_0) - \lambda (y_\vep - x_0)|^2\over \vep}\leq u_\delta(x_\vep, t_\vep)-u_\delta(x_0, t_0)-\lambda v(y_\vep, t_\vep)+\lambda v(x_0, t_0)- {1 \over T-t_0},
\]
which implies the existence of $(\hat{z}, \hat{t})\in (-\lambda (b + x_0), \lambda(b-x_0) ) \times [0, T)$ such that along a subsequence
\[
x_\vep-x_0,  \lambda (y_\vep - x_0) \to \hat{z}, \quad t_\vep\to \hat{t}
\]
as $\vep\to 0$. Hereafter we still index the converging subsequence by $\vep$ for convenience of notation. 

By the convergence property, we have $x_\vep, y_\vep \in (-b,b)$ when $\vep > 0$ is sufficiently small. 
Moreover, since it holds that 
\[ \Phi_{\vep}(\lambda (x-x_0) + x_0, x, t) = \Psi_{\delta, \lambda}(x,t) \quad \text{for} \; \; (x,t) \in Q_0^T, \]
by the semicontinuity of $u_\delta$ and $v$, we can use \eqref{remove lip5} to get 
\[
\Psi_{\delta, \lambda}\left(\frac{\hat{z}}{\lambda} + x_0, \hat{t}\right) \geq \limsup_{\vep\to 0}\Phi_{\vep}(x_\vep, y_\vep, t_\vep)\geq \sup_{Q_0^T} \Psi_{\delta, \lambda} >0,
\]
which yields $\hat{t}\geq 2\sqrt{2C_{T+1}\delta}$ from \eqref{range lambda3} and thus $t_\vep>\sqrt{2C_{T+1}\delta}$ when $\vep>0$ is sufficiently small. 
We fix such $\vep > 0$. 

We next apply the Crandall-Ishii lemma (cf. \cite[Theorem 8.3]{CIL}) to find, for any $\sigma>0$ small, 
$(\tau_1, p_1,  z_1)\in \ol{P}^{2, +} u_\delta(x_\vep, t_\vep)$ and $(\tau_2, p_2, z_2)\in \ol{P}^{2, -} v(y_\vep, t_\vep)$ satisfying 
\beq\label{semijet00}
 \tau_1={ \lambda \tau_2}+{1\over (T-t_\vep)^2},
\eeq
\beq\label{semijet1}
p_1=p_2={2((x_\vep-x_0)-\lambda (y_\vep-x_0))\over \vep}, 
\eeq
and 
\beq \label{semijet2}
\begin{pmatrix}
z_1 & 0\\
0 & -\lambda z_2
\end{pmatrix} \le
\dfrac{2}{\vep}\begin{pmatrix} 1 & -\lambda \\ -\lambda & \lambda^2 \end{pmatrix}+{4\sigma\over \vep^2}\begin{pmatrix} 1 & -\lambda \\ -\lambda & \lambda^2 \end{pmatrix}^2. 
\eeq
The time Lipschitz regularity of $u_\delta$ as in \eqref{time lip} yields 
\begin{equation}\label{bdd-jet-t}
|\tau_1|\leq L_\delta,
\end{equation}
and therefore by \eqref{range lambda1} and \eqref{semijet00} again
\beq\label{semijet0}
\tau_1-\tau_2\geq -(\lambda^{-1} - 1)L_\delta+{1\over \lambda(T-t_\vep)^2}\geq {1 \over 2T^2}.
\eeq
It follows from the monotonicity of $f^{-1}$, \eqref{range lambda2}, \eqref{bdd-jet-t} and \eqref{semijet0} that 
\begin{equation} \label{time contra}
\lambda f^{-1}(\tau_1)-f^{-1}(\tau_2) \ge \min_{|s| \le L_\delta} \left\{\lambda f^{-1}(s) - f^{-1}\left(s - \dfrac{1}{2T^2}\right)\right\} > 0. 
\end{equation}
Multiplying both sides of \eqref{semijet2} by the vector $(\lambda, 1)$ from left and from right, we deduce that 
\beq\label{semijet3}
\lambda z_1\leq z_2.
\eeq

We then adopt the definition of viscosity sub- and super-solutions for $u_\delta$ and $v$ respectively to deduce that 
\[
\tau_1+F(p_1, z_1)\leq 0
\]
and
\[
\tau_2+F(p_2, z_2)\geq 0,
\]
which yields by the assumption (A1), \eqref{semijet1} and \eqref{semijet3}
\[
\lambda f^{-1}(\tau_1)-f^{-1}(\tau_2) \leq \lambda g(p_1)z_1 - g(p_2)z_2 \le 0. 
\]
This is a contradiction to \eqref{time contra}. 
\end{proof}

\section{Existence and nonexistence of solutions}\label{sec:exist}

In this section, we prove the existence and non-existence theories as in Theorem \ref{thm:existence}--\ref{thm:non-existence}. 
We adopt Perron's method when we prove the existence of the solution, and thus, the solution will be constructed as 
\[ u(x,t) := \sup\{\phi(x,t): \phi \; \text{is a sub-solution to \eqref{flow eq} satisfying} \; w \le \phi \le v\}, \]
where $w$ and $v$ are sub- and super-solution to \eqref{flow eq}--\eqref{initial}. 
Furthermore, we also need to construct a sub-solution diverging at the boundary points $x=\pm b$ for any $t>0$ to prove that $u$ satisfies the boundary condition \eqref{bdry}. 
Whether such sub-solutions and super-solution can be constructed depends on the functions $f$ and $g$ in the equation \eqref{flow eq} and the boundedness of the function $u_0$. 

When we prove the non-existence of the solution, we will prove that if we assume the existence of the solution $u$ then $u(x,t) = \infty$ for any $(x,t) \in (-b, b) \times (0, \infty)$. 
The instantaneous interior blowup described above can be proved by constructing a sequence of sub-solutions that are uniformly bounded at initial time but diverges to $\infty$ at any interior points for $t>0$. 
The constructability of such sequence also depends on $f$ and $g$. 

In the following subsections, we divide the cases into (1) $\alpha>1$ under the assumption (B2); (2) $1 < \alpha \le 2$ under the assumption (B1); (3) $\alpha = 1$ or $\alpha < 1$ and $\beta < \frac{1}{1-\alpha}$; and (4) $\alpha < 1$ and $\beta \ge \frac{1}{1-\alpha}$ for the constants $\alpha$ in the assumption (A2) and $\beta$ in the assumption (A3). 
We then give a proof for each existence and non-existence theorem below. 

\subsection{Existence for $\alpha >1$ under the assumption (B2)} \label{subsec:exists1}

When $\alpha > 1$, as we calculated in the proof of Lemma \ref{lem:initial-conti}, due to the assumption (A2), we have 
\begin{align*} 
&g(h_x(x)) h_{xx}(x) \approx (b-x)^{\gamma_+(\alpha -1) + \alpha -2} \quad \text{as} \; \; x \to b, \\
&g(h_x(x)) h_{xx}(x) \approx (b+x)^{\gamma_-(\alpha -1) + \alpha -2} \quad \text{as} \; \; x \to -b
\end{align*}
for $h \in C^2((-b, b))$ satisfying the condition (h1) for some constants $\gamma_\pm \ge 0$ and $D_\pm >0$. 
Therefore, if $\gamma_\pm \ge \max\{\frac{2 - \alpha}{\alpha-1}, 0\}$, we have the boundedness of $g(h_x)h_{xx}$ on $(-b,b)$. 
In this case, we can construct a traveling wave super-solution $v$ to \eqref{flow eq} and \eqref{bdry} can be constructed by letting 
\[ M:= \sup_{x \in (-b, b)} g(h_x(x)) h_{xx}(x) > 0, \quad v(x,t) := h(x) + f(2M) t \; \; \text{for} \; \; (x,t) \in Q_0. \] 
We here note that the speed constant $f(2M)$ is chosen so that Theorem \ref{thm:com-conti} can be applied. 

In the proof of the existence theory when $\alpha > 1$, to apply Perron's method, we first approximate $u_0$ from above with $v_{\varepsilon, 0} \in C^\infty((-b, b))$ satisfying the condition (h1) with some constants $\gamma_\pm \ge \max\{\frac{2 - \alpha}{\alpha-1}, 0\}$ and $D_\pm >0$, and we then construct a traveling wave super-solution $v_\varepsilon$ by the above strategy. 
For the approximation to $u_0$ from below, we use bounded and smooth functions $w_{\varepsilon, 0} \in C^\infty([-b, b])$. 
This allows us to construct a traveling wave sub-solution $w_\varepsilon$ as well as the super-solution $v_\varepsilon$. 
Since we can obtain the sequence of solution $u_\varepsilon$ between $v_\varepsilon$ and $w_\varepsilon$ due to Perron's method, we can apply the standard stability argument for viscosity solutions to construct a solution $u$ to \eqref{flow eq}--\eqref{initial}. 

We now give a proof of Theorem \ref{thm:existence} and Theorem \ref{thm:existence2} under the assumption (B2). 

\begin{proof}[Proof of Theorem \ref{thm:existence} and Theorem \ref{thm:existence2} under the assumption (B2)]
We prove the existence of the solution stated in Theorem \ref{thm:existence} and Theorem \ref{thm:existence2} assuming that the initial function $u_0$ satisfies (B2). 
The continuity and the uniqueness of the solution under the assumption (B3) with $\gamma_\pm \ge \max\{\frac{2-\alpha}{\alpha-1}, 0\}$ can be proved easily by using the comparison result in Theorem \ref{thm:com2}, and we thus omit its proof. 

In order to prove the existence theory, we apply the Perron method. 
Therefore, we have to construct sequences of super- and sub-solutions starting from a perturbation of $u_0$. 
We first construct a sequence of super-solutions $v_\varepsilon$ for $\varepsilon > 0$. 

Let $\hat{\gamma} := \max\{\gamma, \frac{2-\alpha}{\alpha - 1}\} > 0$, where $\gamma$ is the constant in (B2). 
Let also $C > 0$ be a constant satisfying 
\[ C \ge \max\left\{\limsup_{x \to b} u_0(x) (b-x)^{\hat{\gamma}}, \limsup_{x \to -b} u_0(x) (b+x)^{\hat{\gamma}} \right\} \]
and $h \in C^2((-b, b))$ be a function satisfying 
\[ h(x) = \begin{cases}
C(b-x)^{-\hat{\gamma}} & \text{for} \quad x \in [b_0, b), \\
C(b+x)^{-\hat{\gamma}} & \text{for} \quad x \in (-b, -b_0] 
\end{cases} \]
for some $b_0 \in (0, b)$. 
Then, due to \eqref{initial div}, we have $u_0(x) - h(x) \to -\infty$ as $x \to \pm b$. 
Therefore, for $\varepsilon > 0$, we can choose $\tilde{v}_{\varepsilon, 0} \in C(\mathbb{R})$ as
\[ \tilde{v}_{\varepsilon, 0}(x) = \max \left\{u_0(x) - h(x), - \frac{1}{\varepsilon}\right\} \; \; \text{for} \; \; x \in (-b, b), \quad \tilde{v}_{\varepsilon, 0}(x) = -\frac{1}{\varepsilon} \; \; \text{for} \; \; |x| \ge b \]
and $\hat{v}_{\varepsilon, 0} \in C^\infty(\mathbb{R})$ satisfying 
\[ |\hat{v}_{\varepsilon, 0}(x) - \tilde{v}_{\varepsilon, 0}(x)| \le \varepsilon \quad \text{for} \; \; x \in [-b, b]. \]
Letting 
\[ v_{\varepsilon, 0}(x) := \hat{v}_{\varepsilon, 0}(x) + h(x) \quad \text{for} \; \; x \in (-b, b), \]
we have
\begin{align}
&\lim_{\varepsilon \to 0} v_{\varepsilon, 0}(x) = u_0(x) \quad \text{for} \; \; x \in (-b, b), \label{thm1-vep1}\\
&u_0(x) \le v_{\varepsilon, 0}(x) + \varepsilon \quad \text{for} \; \; x \in (-b, b), \; \; \varepsilon > 0, \label{thm1-vep2}\\
&v_{\varepsilon_2, 0}(x) \le v_{\varepsilon_1, 0}(x) + \varepsilon_1 + \varepsilon_2 \quad \text{for} \; \; x \in (-b, b), \; \; \varepsilon_1 \ge \varepsilon_2 > 0. \label{thm1-vep3}
\end{align}
We also note that 
\begin{align*} 
&v_{\varepsilon, 0}(x) \approx (\min\{b-x, b+x\})^{-\hat{\gamma}}, \\ 
&|(v_{\varepsilon, 0})_x(x)| \approx (\min\{b-x, b+x\})^{-\hat{\gamma}-1}, \\
&(v_{\varepsilon, 0})_{xx}(x) \approx (\min\{b-x, b+x\})^{-\hat{\gamma}-2} 
\end{align*}
near the boundary points $x=\pm b$. 
Therefore, due to $\hat{\gamma} \ge \max\{\frac{2-\alpha}{\alpha-1}, 0\}$, we can see by a similar argument as in the proof of Lemma \ref{lem:initial-conti} that there exists $M_\varepsilon>0$ such that 
\[ f\left(2 g((v_{\varepsilon, 0})_x(x))(v_{\varepsilon, 0})_{xx}(x)\right) \le M_\varepsilon \quad \text{for} \; \; x \in (-b, b). \]
Thus, letting 
\begin{equation}\label{const-v-ep-4}
v_\varepsilon(x,t) := v_{\varepsilon, 0}(x) + \varepsilon + M_\varepsilon t \quad  \text{for} \; \;  (x,t) \in (-b, b) \times [0, \infty), 
\end{equation}
we have 
\[ (v_\varepsilon)_t \ge \max\{f(2g((v_\varepsilon)_x)(v_\varepsilon)_{xx}), 0\} \quad \text{on} \; \; (-b, b) \times (0, \infty). \]

We next construct a sequence of sub-solutions $w_\varepsilon$. 
From $\lim_{x \to \pm b} u_0(x) = \infty$, we can choose $\tilde{w}_{\varepsilon, 0} \in C(\mathbb{R})$ as  
\[ \tilde{w}_{\varepsilon, 0}(x) = \min\left\{u_0, \frac{1}{\varepsilon} \right\} \; \; \text{for} \; \; x \in (-b, b), \quad \tilde{w}_{\varepsilon, 0}(x) = \frac{1}{\varepsilon} \; \; \text{for} \; \; |x| \ge b \]
and also $w_{\varepsilon, 0} \in C^\infty(\mathbb{R})$ satisfying 
\[ |w_{\varepsilon, 0}(x) - \tilde{w}_{\varepsilon, 0}(x)| \le \varepsilon \quad \text{for} \; \; x \in [-b, b]. \]
Then, we can define 
\[ m_\varepsilon := \min_{x \in [-b, b]} f(g((w_{\varepsilon, 0})_x(x)) (w_{\varepsilon, 0})_{xx}(x)), \]
and we can construct a sub-solution $w_\varepsilon$ as 
\[ w_\varepsilon(x,t) := w_{\varepsilon, 0}(x) - \varepsilon + m_\varepsilon t \quad \text{for} \; \; (x,t) \in (-b, b) \times [0, \infty). \]
We note that it holds that 
\begin{align}
& \lim_{\varepsilon \to 0} w_{\varepsilon, 0}(x) = u_0(x) \quad \text{for} \; \; x \in (-b, b), \label{thm1-wep1} \\
& w_{\varepsilon, 0}(x) \le u_0(x) + \varepsilon \quad \text{for} \; \; x \in (-b, b), \; \; \varepsilon > 0, \label{thm1-wep2} \\
& w_{\varepsilon_1, 0} (x) \le w_{\varepsilon_2, 0}(x) + \varepsilon_1 + \varepsilon_2 \quad \text{for} \; \; x \in (-b, b), \; \; \varepsilon_1 \ge \varepsilon_2 > 0. \label{thm1-wep3}
\end{align}

We now define $u_\varepsilon : (-b, b) \times [0, \infty) \to \mathbb{R}$ as 
\[ u_\varepsilon(x,t) := \sup\{u(x,t): u \; \text{is a sub-solution to \eqref{flow eq} satisfying} \; w_\varepsilon \le u \le v_\varepsilon\}. \]
Notice that \eqref{thm1-vep2} and \eqref{thm1-wep2} yield 
\begin{equation}\label{thm1-com-initial} 
w_\varepsilon(\cdot, 0) = w_{\varepsilon, 0} - \varepsilon \le u_0 \le v_{\varepsilon, 0} + \varepsilon = v_\varepsilon(\cdot, 0)  \quad \text{on} \; \; (-b, b) 
\end{equation}
and the comparison result in Theorem \ref{thm:com1} further implies $w_\varepsilon \le v_\varepsilon$ on $(-b, b) \times [0, \infty)$. 
Then, the standard Perron method (cf.\ \cite{CIL}) yields that $u_\varepsilon$ is a viscosity solution to \eqref{flow eq}. 
Due to \eqref{thm1-vep3}, \eqref{thm1-wep3} and the comparison results, we have 
\[ w_{\varepsilon_1}(x,t) - 2\varepsilon_1 \le u_{\varepsilon_2}(x,t) \le (u_{\varepsilon_2})^*(x,t) \le v_{\varepsilon_1}(x,t) + 2\varepsilon_2 \quad \text{for} \; \; (x,t) \in (-b, b) \times [0,\infty) \]
if $\varepsilon_1 \ge \varepsilon_2 > 0$. 
Therefore, we can take the relaxed half limit of $(u_\varepsilon)^*$ as 
\[ u(x,t) := \limsups_{\varepsilon \to 0} (u_\varepsilon)^*(x,t) = \lim_{\delta \to 0} \sup \{(u_\varepsilon)^*(y,s) : (y,s) \in Q_0, |x-y| + |t-s| + \varepsilon \le \delta\} \]
for $(x,t) \in (-b, b) \times [0, \infty) = Q_0$. 
By the standard stability result (cf.\ \cite{CIL}) for viscosity solution, we can see that $u$ is a sub-solution to \eqref{flow eq}. 
We can also obtain 
\begin{equation}\label{thm1-order} 
w_\varepsilon(x,t)  \le u(x,t) \le v_\varepsilon(x,t)  \quad \text{for} \; \; (x, t) \in (-b, b) \times [0, \infty), \; \; \varepsilon > 0, 
\end{equation}
which yields by \eqref{thm1-vep1} and \eqref{thm1-wep1} 
\[ u_*(\cdot, 0) = u(\cdot, 0) =u^*(\cdot, 0) = u_0 \quad \text{in} \; \; (-b, b). \]
Therefore, it is sufficient to prove that $u_*$ is a super-solution to \eqref{flow eq} and \eqref{bdry} on $Q_0$. 

Letting 
\begin{equation}\label{def-limit-vw} 
v(x,t) := \inf \{v_\varepsilon(x,t): \varepsilon > 0\}, \quad w(x,t) := \sup \{w_\varepsilon(x,y): \varepsilon > 0\}, 
\end{equation}
\eqref{thm1-order} yields $w \le u \le v$ on $(-b,b) \times [0, \infty)$. 
We note that $w^*$ is also a sub-solution to \eqref{flow eq} due to the standard stability result. 
At the initial time $t=0$, the properties \eqref{thm1-vep1}, \eqref{thm1-vep2}, \eqref{thm1-wep1} and \eqref{thm1-wep2} yield $v(\cdot, 0) = w(\cdot, 0) = u_0$ on $(-b, b)$. 
Therefore, by the comparison result in Theorem \ref{thm:com-conti}, we have $w \le v_\varepsilon$ on $Q_0$ for any $\varepsilon > 0$. 
Due to the continuity of $v_\varepsilon$, we get $w^* \le v_\varepsilon$ on $Q_0$, which yields 
\begin{equation}\label{order-limit-vw} 
w^*(x,t) \le v(x,t) \quad \text{for} \; \; (x,t) \in Q_0 
\end{equation}
by the definition of $v$. 

Furthermore, we can see that 
\begin{align}\label{thm1-sup-con} 
u(x,t) = \sup \{\phi(x,t) : \phi \; \text{is a sub-solution to \eqref{flow eq} satisfying} \; w \le \phi \le v\}. 
\end{align}
Indeed, assume that there exists a sub-solution $\phi$ to \eqref{flow eq} satisfying $\phi \le v$ and $\phi(y,s) > u(y,s)$ at some $(y, s) \in (-b, b) \times [0, \infty)$. 
Then, since $\phi(\cdot, 0) = u_0$, the estimate \eqref{thm1-com-initial} and the comparison results yield $w_\varepsilon \le \phi \le v_\varepsilon$ on $(-b, b) \times [0, \infty)$ for $\varepsilon > 0$. 
Therefore, since $\phi$ is a sub-solution to \eqref{flow eq}, due to the definition of $u_\varepsilon$, we have $\phi \le u_\varepsilon \le (u_\varepsilon)^*$ on $Q_0$. 
On the other hand, due to the definition of $u$, we have 
\[ \liminf_{\varepsilon \to 0} \phi(y,s) - (u_\varepsilon)^*(y,s) \ge \phi(y,s) - u(y,s) > 0, \]
which contradicts to $\phi \le (u_\varepsilon)^*$ on $Q_0$ for $\varepsilon > 0$. 
Therefore, the claim holds; thus the standard Perron method implies that $u_*$ is a super-solution to \eqref{flow eq}. 

We next prove that the boundary condition $\lim_{x \to \pm b} u_*(x,t) = \infty$ holds for $t > 0$. 
Taking the lower convex envelope of $u_0$ as 
\[ \tilde{u}_0(x) := \sup\{\phi(x) : \phi \; \text{is convex and} \; \phi \le u_0 \; \text{on} \; (-b, b)\}. \]
Then, $\tilde{u}_0$ is continuous, convex and satisfies 
\[ \lim_{x \to \pm b} \tilde{u}_0 = \infty. \]
Therefore, letting 
\[ \tilde{u}(x, t) := \tilde{u}_0 (x) \quad \text{for} \; \; (x,t) \in Q_0, \]
due to \eqref{thm1-vep2}, the comparison result in Theorem \ref{thm:com-conti} shows that $\tilde{u} \le v_{\varepsilon}$ on $Q_0$. 
By the definition of $v$, we obtain $\tilde{u} \le v$ on $Q_0$. 
Therefore, due to \eqref{order-limit-vw}, letting 
\[ \tilde{w}(x,t) := \max\{w^*(x,t), \tilde{u}(x,t)\} \quad \text{for} \; \; (x,t) \in Q_0, \]
we can see that $\tilde{w}$ is a sub-solution to \eqref{flow eq} satisfying $w \le \tilde{w} \le v$ on $Q_0$. 
Therefore, the property \eqref{thm1-sup-con} and the definition of $\tilde{w}$ yield $\tilde{u} \le \tilde{w} \le u$ on $Q_0$. 
The continuity of $\tilde{u}$ thus shows that $\tilde{u} \le u_*$ on $Q_0$, which yields $\lim_{x \to \pm b} u_*(x,t) = \infty$ for $t \in [0, \infty)$.
\end{proof}

\subsection{Existence for $1<\alpha \le 2$ under the assumption (B1)} \label{subsec:exists2}

In this section, we prove the existence of the solution when $1 < \alpha \le 2$ under the assumption (B1). 
The strategy to prove the existence is similar to it in Section \ref{subsec:exists1}. 
However, in the proof of Theorem \ref{thm:existence} and Theorem \ref{thm:existence2} under the assumption (B2), the divergence property 
\[ \lim_{x \to \pm b} u_0(x) = \infty \]
have been used to construct sub-solutions $w_\varepsilon$ and $\tilde{u}$. 

Under the assumption (B1), we approximate $u_0$ from below with a bounded and smooth function $w_{0, \varepsilon} \in C^\infty([-b,b])$ due to the continuity of $u_0$ up to the boundary points $x=\pm b$ and construct a traveling wave sub-solution $w_\varepsilon$ from $w_{0, \varepsilon}$. 
In order to prove that a solution constructed by Perron's method and the standard stability argument for viscosity solutions satisfies the boundary condition \eqref{bdry}, we will use the sequence of sub-solutions $\tilde{u}_k$ constructed in \cite[Section 4.2]{KL} instead of $\tilde{u}$ in Section \ref{subsec:exists1}. 
We note that these constructions of the above sub-solutions $w_\varepsilon$ and $\tilde{u}_k$ can be applied also when $\alpha = 1$ or $\alpha < 1$ and $\beta < \frac{1}{1-\alpha}$, and thus, we will use these constructions in Section \ref{subsec:exist3}. 
We here summarize the definition and the properties of the sequence $\tilde{u}_k$ constructed in \cite[Section 4.2]{KL} below. 

We first state the definition of the sequence. 

\begin{defi}[{\cite[Section 4.2]{KL}}]\label{defi:sub by KL}
Let $b>0$. Assume $f$ and $g$ satisfy (A1) and (A2) with $\alpha \le 2$. 
Then, there exist $s_0 > 0$ and $M>0$ such that 
\[ g(s) \ge 2M|s|^{-2} \quad \text{for} \; \; |s| \ge s_0. \]
For sufficiently large $k>0$, let 
\[ r_k := \sqrt{\frac{1+k^2}{k^2}} b \]
and $y_k \in (0, e^{-1})$ is chosen to satisfy 
\[ \frac{1}{y_k \log y_k} = -k. \]
Define $x(k,t)$ by 
\[ x(k,t) = \exp(-\exp \{f(-M \log y_k) t + \log(-\log y_k)\}) + b - y_k \quad \text{for} \; \; t \ge 0.  \]
Define a function $\tilde{u}_k: [-b, b] \times [0, \infty) \to \mathbb{R}$ by 
\[ \tilde{u}_k(x,t) := \begin{cases}
\log\left( \frac{\log\{x(k,t) - x + y_k\}}{\log y_k} \right) - \sqrt{r_k^2 - x(k, t)^2} & \text{for} \; \; x \in (x(k, t), b], \\
-\sqrt{r_k^2 - x^2} & \text{for} \; \; x \in [-x(k, t), x(k, t)], \\
\log \left( \frac{\log\{x(k,t) + x + y_k\}}{\log y_k} \right) - \sqrt{r_k^2 - x(k, t)^2} & \text{for} \; \; x \in [-b, -x(k,t))
\end{cases} \] 
for $t \ge 0$. 
\end{defi}

We then can obtain the following properties on the sequence. 

\begin{lem}[{\cite[Lemma 4.8]{KL}}]\label{lem:div-KL}
Let $b>0$. Assume $f$ and $g$ satisfy (A1) and (A2) with $\alpha \le 2$. 
Then, the function $\tilde{u}_k$ defined by Definition \ref{defi:sub by KL} is continuous in $\overline{Q}_0$ and a sub-solution to \eqref{flow eq} for sufficiently large $k>0$. 
Furthermore, $\tilde{u}_k(\cdot, 0)$ is uniformly bounded on $[-b, b]$ and 
\[ \lim_{k \to \infty} \tilde{u}_k(\pm b, t) = \infty \quad \text{for} \; \; t > 0. \]
\end{lem}

\begin{rmk}
Since $\tilde{u}_k$ is of class $C(\overline{Q}_0)$, the function $\tilde{u}_k(\cdot, t)$ is bounded on $[-b, b]$ for any $t>0$. 
Therefore, we can apply the comparison result in Theorem \ref{thm:com1} to $\tilde{u}_k$. 
We can also easily see by the definition of $\tilde{u}_k$ that 
\[ \lim_{k \to \infty} \tilde{u}_k(x, 0) = \lim_{k \to \infty} -\sqrt{b_k^2 - x^2} = -\sqrt{b^2 - x^2} \quad \text{for} \; \; x \in [-b, b]. \]
\end{rmk}

Let us now prove Theorem \ref{thm:existence2} under the assumption (B1). 

\begin{proof}[Proof of Theorem \ref{thm:existence2} under the assumption (B1)]
We prove the existence of the solution stated in Theorem \ref{thm:existence2} assuming that the initial function $u_0$ satisfies (B1). 
Let $\hat{\gamma} := \frac{2-\alpha}{\alpha - 1} + 1 > 0$. 
Since $u_0$ is of class $C([-b, b])$, $u_0$ satisfies 
\[ \lim_{x \to \pm b} u_0(x) (b\mp x)^{\hat{\gamma}} = 0. \]
Then, for $\varepsilon > 0$, the super-solution $v_\varepsilon$ to \eqref{flow eq} and \eqref{bdry} defined by \eqref{const-v-ep-4} can be constructed in the same way in the proof of Theorem \ref{thm:existence} and Theorem \ref{thm:existence2} under the assumption (B2). 

We next construct a sequence of sub-solutions $w_\varepsilon$. 
Due to $u_0 \in C([-b, b])$, we can choose $w_{\varepsilon, 0} \in C^\infty([-b, b])$ satisfying 
\[ |w_{\varepsilon, 0}(x) - u_0(x)| \le \varepsilon \quad \text{for} \; \; x \in [-b, b]. \]
Then, we can define 
\[ m_\varepsilon := \min_{x \in [-b, b]} f(g((w_{\varepsilon, 0})_x(x)) (w_{\varepsilon, 0})_{xx}(x)), \]
and we can construct a sub-solution $w_\varepsilon$ as 
\[ w_\varepsilon(x,t) := w_{\varepsilon, 0} - \varepsilon + m_\varepsilon t \quad \text{for} \; \; (x,t) \in (-b, b) \times [0, \infty). \]
We can see that $w_\varepsilon$ satisfies \eqref{thm1-wep1}--\eqref{thm1-wep3} by simple calculations. 

Therefore, we can obtain a solution $u_\varepsilon$ between $v_\varepsilon$ and $w_\varepsilon$ as 
\[ u_\varepsilon(x,t) := \sup\{u(x,t): u \; \text{is a sub-solution to \eqref{flow eq} satisfying} \; w_\varepsilon \le u \le v_\varepsilon\}, \]
and also, by the standard stability argument as in the proof of Theorem \ref{thm:existence} and Theorem \ref{thm:existence2} under the assumption (B2), we can construct a solution $u$ to \eqref{flow eq} satisfying \eqref{initial} in the viscosity sense by taking the relaxed half limit of $(u_\varepsilon)^*$. 
We note that the solution $u$ satisfies 
\begin{equation} \label{thm2-sup-con}
u(x,t) = \sup \{\phi(x,t) : \phi \; \text{is a sub-solution to \eqref{flow eq} satisfying} \; w \le \phi \le v\}, 
\end{equation}
where $v$ and $w$ are functions given by \eqref{def-limit-vw} for $v_\varepsilon$ and $w_\varepsilon$ of this proof. 

We next prove that the boundary condition $\lim_{x \to \pm b} u_*(x,t) = \infty$ holds for $t>0$. 
For $k > 0$ large, let $\tilde{u}_k$ be the sequence of sub-solutions defined by Definition \ref{defi:sub by KL}. 
Due to the uniformly boundedness of $\tilde{u}_k(\cdot, 0)$, there exists $D \in \mathbb{R}$ such that $\tilde{u}_k(\cdot, 0) + D \le u_0$ on $[-b, b]$ for any $k > 0$ large. 
Therefore, the comparison result in Theorem \ref{thm:com-conti} (or Theorem \ref{thm:com1}) shows that $\tilde{u}_k + D \le v_{\varepsilon}$ on $Q_0$.
By the definition of $v$, we have $\tilde{u}_k + D \le v$ on $Q_0$. 
Thus, letting $\tilde{w}_k(x,t) := \max\{w^*(x,t), \tilde{u}_k(x,t) + D\}$ for $(x,t) \in Q_0$, we have $\tilde{u}_k + D \le \tilde{w} \le u$ on $Q_0$ due to \eqref{thm2-sup-con}. 
Since $\tilde{u}_k$ is continuous on $\overline{Q}_0$, we have 
\[ \tilde{u}_k (\pm b, t) + D \le \liminf_{x \to \pm b} u_*(x,t) \quad \text{for} \; \; t > 0 \; \text{and}  \; k > 0 \; \text{large}. \]
Letting $k \to \infty$, Lemma \ref{lem:div-KL} yields that $u_*$ satisfies the boundary condition. 
\end{proof}

\subsection{Existence for $\alpha = 1$ or $\alpha < 1$ and $\beta < \frac{1}{1-\alpha}$} \label{subsec:exist3}

In this section, we prove the existence result (b) in Theorem \ref{thm:non-existence}. 
The proof method in this section is almost identical to it in Section \ref{subsec:exists1} and \ref{subsec:exists2}, differing only in the construction of the super-solutions $v_\varepsilon$. 
Sub-solutions $w_\varepsilon$ to construct the solution by Perron's method and the standard stability argument for viscosity solutions and sub-solution(s) $\tilde{u}$ (or $\tilde{u}_k$) to prove that the solution satisfies the boundary condition \eqref{bdry} can be constructed as in Section \ref{subsec:exists1} when (B2) is assumed and as in Section \ref{subsec:exists2} when (B1) is assumed. 
Therefore, we give a heuristically argument only to construct the super-solutions $v_\varepsilon$ below and we only rigorously construct the super-solutions in the proof to omit the rest of arguments. 

In the proof, we will construct a super-solution to \eqref{flow eq} and \eqref{bdry} formed by, for example near the boundary point $x=b$, $v(x,t) = (b-x)^{-L(t)}$ with increasing function $L(t)$. 
Substituting this function into the inequality $v_t \ge f(g(v_x)v_{xx})$, due to the assumptions (A2) and (A3), we obtain the approximated inequality 
\[ \frac{L'(t)}{(L(t))^{\beta(1-\alpha)} (L(t) + 1)^\beta} \gtrsim \frac{1}{\log \frac{1}{b-x}} (b-x)^{L(t)(1-\beta(1-\alpha)) - \beta(2-\alpha)} \quad \text{if} \; 0 < b-x \ll 1. \]
From $1 - \beta(1-\alpha) > 0$, the right hand side is bounded from above if $L(t)$ is sufficiently large, and thus if $L(t)$ satisfies the ordinary differential inequality 
\[ L'(t) \ge M (L(t))^{\beta(1-\alpha)} (L(t) + 1)^\beta \]
for some $M>0$, then we can see that $v(x,t)$ is a super-solution near the boundary point $x=b$. 
A similar argument can be applied near the opposite $x=-b$. 

To construct the solution to the initial value problem using Perron's method, we need to modify the above discussion to construct a super-solution starting from a given initial function. 
We first construct a family of super-solutions so that we can choose initial function of the super-solution later, which formed as follows: 

\begin{defi}\label{definition-super1}
Let $T>0$. 
For functions $v_0 \in C^2([-b, b])$ and $L \in C^1([0, T])$, and constants $\mu \ge 0$ and $\nu, c>0$, we define $v(x,t)$ as 
\begin{align*} 
&v(x,t) := \\
&\left\{ \begin{aligned}
&\begin{aligned}
(L(t))^\mu(2 - \tfrac{2x}{b})^{-L(t)} - (L(0))^\mu \left(\tfrac{2}{3}\right)^{-L(0)} +&\; v_0(x) + ct +\tfrac{1}{T-t} - \tfrac{1}{T} \\
& \text{for} \; \tfrac{2b}{3} \le x < b, \; 0 \le t < T, 
\end{aligned} \\
&\begin{aligned}
-\tfrac{1}{\sqrt{\nu}} \sqrt{r_\nu^2 - x^2} + \tfrac{2b}{3\nu \sqrt{\nu}} + (L(t))^\mu \left(\tfrac{2}{3}\right)^{-L(t)}&\; - (L(0))^\mu \left(\tfrac{2}{3} \right)^{L(0)} + v_0(x) + ct + \tfrac{1}{T-t} - \tfrac{1}{T}  \\
&\text{for} \; -\tfrac{2b}{3} < x < \tfrac{2b}{3}, \; 0 \le t < T,
\end{aligned}\\
&\begin{aligned}
(L(t))^\mu(2 + \tfrac{2x}{b})^{-L(t)} - (L(0))^\mu \left(\tfrac{2}{3}\right)^{-L(0)} +&\; v_0(x) + ct +\tfrac{1}{T-t} - \tfrac{1}{T} \\
& \text{for} \; - b < x \le - \tfrac{2b}{3}, \; 0 \le t < T,
\end{aligned}
\end{aligned} \right. 
\end{align*}
where $r_\nu := \frac{2b}{3}\sqrt{1+\frac{1}{\nu^2}}$. 
\end{defi}

\begin{rmk}
We here note the value of $v(x, 0)$ at $x = \pm \frac{2b}{3}$ below. 
At the point $x= \frac{2b}{3}$, we have 
\[ 2 - \frac{2x}{b} = \frac{2}{3}, \quad \sqrt{r_\nu^2 - x^2} = \frac{2b}{3\nu}, \]
which yields the continuity of $v(\cdot, 0)$ at $x=\frac{2b}{3}$ and $v(\frac{2b}{3}, 0) = v_0(\frac{2b}{3})$. 
By calculating similarly at the opposite point $x = - \frac{2b}{3}$, we can see that $v(\cdot, 0)$ is continuous at $x=-\frac{2b}{3}$ and $v(- \frac{2b}{3}, 0) = v_0(-\frac{2b}{3})$. 
We can also see that $v$ is continuous also at $(\pm \frac{2b}{3}, t)$ for $t > 0$. 
\end{rmk}

We have to choose $T, \nu, c>0$, $\mu \ge 0$ and $L(t)$ so that $v$ is a super-solution as follows: 

\begin{prop}\label{prop:upper-esti-smalla}
Let $b > 0$ and $v_0 \in C^2([-b, b])$. 
Assume that a function $f$ satisfies (A1) and (A3) and a function g satisfies (A2). 
Assume also $\alpha = 1$ or $\alpha < 1$ and $0 < \beta < \frac{1}{1-\alpha}$. 
Then, letting 
\begin{equation}\label{defi-mu} 
\mu = \max \left\{0, \frac{\beta(2-\alpha) - 1}{1 - \beta(1-\alpha)}\right\}, 
\end{equation}
for any constant $L_0$ and $\nu$ satisfying 
\begin{equation}\label{cond-del0} 
L_0 > \frac{1-\beta(1-\alpha)}{\beta(2-\alpha)}, \quad \nu > \frac{8b L_0^{2\mu + 2}}{9} \left(\frac{3}{2}\right)^{2L_0+2}
\end{equation} 
there exist constants $T, c> 0$ and a function $L(t)$ such that $L(0) = L_0$ and the function $v$ defined in Definition \ref{definition-super1} satisfies
\begin{equation}\label{aim-sub-ine} 
v_t \ge \max\{f(2g(v_x)v_{xx}), 0\} \quad \text{on} \; \; (-b, b) \times (0, T) 
\end{equation}
in the viscosity sense. 
Furthermore, $T \to \infty$ as $\nu \to \infty$. 
\end{prop}

\begin{proof}
Due to the assumptions (A2) and (A3), there exist constants $L_g, M_{g+}, M_{g-}, L_f, M_f > 0$ such that 
\begin{align}
& M_{g-} |s|^{-\alpha} \le g(s) \le M_{g+} |s|^{-\alpha} \quad \text{for} \; \; |s| \ge L_g, \label{esti-g1}\\
& f(s) \le M_f s^{\beta} \quad \text{for} \; \; s \ge L_f. \label{esti-f1}
\end{align}
We will choose constants $c,T > 0$ and a increasing function $L(t)$ with $L(0) = L_0$ so that $v$ satisfies \eqref{aim-sub-ine} in the viscosity sense. 

We first choose $L(t)$ to satisfy \eqref{aim-sub-ine} at points $x$ near the boundary $x = \pm b$. 
We here let $d(x) := \min\{2-\frac{2x}{b}, 2+\frac{2x}{b}\}$. 
Due to $v_0 \in C^2([-b, b])$ and the definition of $v$, there exist $b_1 \in (\frac{2}{3}b, b)$ and $C_1>0$ depending only on $v_0$ and $L_0$ such that 
\begin{equation}\label{deri-positive} 
|v_x(x,t)| \ge C_1 (d(x))^{-L_0-1} \ge L_g \quad \text{for} \; \; b_1 \le |x| <b, \; \; t \ge 0 
\end{equation}
if $L(t)$ is increasing. 
We here note that \eqref{cond-del0} yields $L_0 > 0$. 
Therefore, \eqref{esti-g1} with $s = v_x(x,t)$ holds for $b_1 \le |x| < b$ and $t \ge 0$, and thus there exist $b_2 \in [b_1, b)$ and $C_2 > 0$ depending only on $v_0$ and $L_0$ such that 
\begin{align*} 
2g(v_x(x,t)) v_{xx}(x,t) \ge&\; 2M_{g-} (v_x (x,t))^{-\alpha} v_{xx}(x,t) \\
\ge&\; 2M_{g-} C_2 (d(x))^{-L_0(1-\alpha) - 2 + \alpha} \ge L_f \quad \text{for} \; \; b_2 \le |x| < b, \; \; t \ge 0
\end{align*}
if $L(t)$ is increasing. 
Notice that $- L_0(1-\alpha) - 2 + \alpha < 0$ from $\alpha \le 1$. 
Therefore, \eqref{esti-f1} with $s = g(v_x(x,t)) v_{xx}(x,t)$ holds for $b_2 \le |x| < b$ and $t \ge 0$. 
Using \eqref{esti-g1} and \eqref{esti-f1} again, we can see that there exist $b_3 \in [b_2, b)$ and $C_3 > 0$ such that 
\begin{align*}
f(2g(v_x(x,t)) v_{xx}(x,t)) \le&\; 2^\beta M_f (M_{g+})^\beta (v_x(x,t))^{-\alpha \beta} (v_{xx}(x,t))^\beta \\
\le&\;  C_3 (L(t))^{\beta(1-\alpha)(1+\mu)}(L(t)+1)^\beta (d(x))^{-L(t)\beta (1-\alpha) - \beta (2 - \alpha)} 
\end{align*}
for $b_3 \le |x| < b$ and $t \ge 0$. 
On the other hand, we have by a simple calculation 
\begin{align*} 
v_t(x,t) =&\; \mu(L(t))^{\mu-1}(d(x))^{-L(t)} + (d(x))^{-L(t)} (L(t))^\mu L'(t) \log \frac{1}{d(x)} + c + \frac{1}{(T-t)^2} \\
\ge&\; (d(x))^{-L(t)} L'(t) (L(t))^\mu \log \frac{3}{2} \ge 0 
\end{align*}
for $\frac{2b}{3} < |x| < b$ if $L(t)$ is increasing and $c$ is positive. 
Therefore, if we choose a increasing function $L(t)$ satisfying 
\[ \frac{L'(t)}{(L(t))^{\beta(1-\alpha)(1+\mu) - \mu} (L(t) + 1)^{\beta}} \ge C_3 (d(x))^{L(t)(1-\beta (1-\alpha)) - \beta (2 - \alpha)} \frac{1}{\log \frac{3}{2}} \]
for $b_3 \le |x| < b$ and $t>0$, then \eqref{aim-sub-ine} holds. 
From $1 - \beta(1-\alpha) > 0$ and \eqref{cond-del0}, if $L(t)$ is increasing, we have $L(t)(1-\beta (1-\alpha)) - \beta (2 - \alpha) > 0$, and it thus holds that 
\[ C_3 (d(x))^{L(t)(1-\beta (1-\alpha)) - \beta (2 - \alpha)} \frac{1}{\log \frac{3}{2}} \le \frac{C_3}{\log \frac{3}{2}} =: C_4 \quad \text{for} \; \; b_3 \le |x| < b, \; \; t > 0. \]
Therefore, we choose a function $L(t)$ as a solution to 
\begin{equation}\label{ode-delta} 
\begin{cases}
L'(t) = C_4 (L(t))^{\beta(1-\alpha)(1+\mu) -\mu} (L(t) + 1)^{\beta} & \text{for} \; t > 0, \\
L(0) = L_0. 
\end{cases} 
\end{equation}
We also let $T>0$ be a positive time satisfying 
\begin{equation}\label{choice-T} 
\nu = \frac{8b (L(T))^{2\mu+2}}{9} \left(\frac{3}{2}\right)^{2L(T) + 2}. 
\end{equation}
Then, \eqref{aim-sub-ine} holds for $b_3 \le |x| < b$ and $0 < t \le T$. 
We note that the ordinary differential equation \eqref{ode-delta} is solvable globally in time due to the choice of $\mu$ in \eqref{defi-mu}. 
Therefore, the constant $T$ diverges to $\infty$ as $\nu \to \infty$. 

We next choose $c>0$ to satisfy \eqref{aim-sub-ine} for $|x| \le b_3$ with $x \neq \pm \frac{2b}{3}$ and $0 < t \le T$. 
By the definition of $v$, we can see that 
\[ C_5 := \sup\left\{ |f(2g(v_x(x,t)) v_{xx}(x,t))| : |x| \le b_3 \; \text{with} \; x \neq \pm \frac{2b}{3}, \; \; 0 < t \le T\right\} < \infty. \]
Letting $c$ be positive and satisfies $c \ge C_5$, since it holds that $v_t(x,t) \ge c$ for $|x| \le b_3$ with $x\neq \frac{2b}{3}$ and $0<t < T$, we can easily see that \eqref{aim-sub-ine} holds at the points. 

At the singular points $x = \pm \frac{2b}{3}$, \eqref{choice-T} yields 
\begin{equation}
\begin{aligned}
& \lim_{x \uparrow -\frac{2b}{3}} v_x(x,t) = - \frac{2(L(t))^{\mu + 1}}{b} \left(\frac{3}{2}\right)^{L(t)+1} > - \frac{3\sqrt{\nu}}{2b} = \lim_{x \downarrow - \frac{2b}{3}} v_x(x,t),  \\
& \lim_{x \uparrow \frac{2b}{3}} v_x(x,t) = \frac{3\sqrt{\nu}}{2b} > \frac{2(L(t))^{\mu+1}}{b} \left(\frac{3}{2}\right)^{L(t)+1} = \lim_{x \downarrow \frac{2b}{3}} v_x(x,t) 
\end{aligned}
\end{equation}
for any $0 \le t < T$. 
Therefore, $v-\phi$ does not attain a local minimum at $x = \pm \frac{2b}{3}$ and $0 < t \le T$ for any $\phi \in C^2(Q)$. 
It shows that $v$ satisfies \eqref{aim-sub-ine} in the viscosity sense in $(-b, b) \times [0, T]$.
\end{proof}

Let us now prove (b) in Theorem \ref{thm:non-existence}. 

\begin{proof}[Proof of (b) in Theorem \ref{thm:non-existence}] 
We construct a sequence of super-solutions $v_\varepsilon$ to apply Perron's method. 
If $u_0 \in C([-b, b])$ then it holds that 
\begin{equation}\label{initial div2} 
\limsup_{x \to b} u_0(x) (b-x)^\gamma < \infty, \quad \limsup_{x \to -b} u_0(x) (b+x)^\gamma < \infty 
\end{equation}
for any $\gamma > 0$, and thus we fix $\gamma > 0$ arbitrary when (B1) is assumed. 
When (B2) is assumed, let $\gamma > 0$ be the constant in the assumption (B2). 
Fix $L_0 > 0$ satisfying the first inequality in \eqref{cond-del0} and $L_0 > \gamma$. 
Define $\psi \in C((-b, b))$ by 
\[ \psi(x) := \begin{cases}
(L_0)^\mu (2-\frac{2x}{b})^{-L_0} - (L_0)^\mu (\frac{2}{3})^{-L_0} & \text{for} \; \; \frac{2b}{3} \le x < b, \\
0 & \text{for} \; \; -\frac{2b}{3} < x< \frac{2b}{3}, \\
(L_0)^\mu (2+\frac{2x}{b})^{-L_0} - (L_0)^\mu (\frac{2}{3})^{-L_0} & \text{for} \; \; -b < x \le -\frac{2b}{3}, 
\end{cases} \]
where $\mu$ is the constant given by \eqref{defi-mu}. 

Due to \eqref{initial div} or \eqref{initial div2}, we have $u_0(x) - \psi(x) \to -\infty$ as $x \to \pm b$. 
Therefore, for $\varepsilon > 0$, we can choose $\tilde{v}_{\varepsilon, 0} \in C(\mathbb{R})$ as
\[ \tilde{v}_{\varepsilon, 0} (x) = \max\left\{u_0(x) - \psi(x), -\frac{1}{\varepsilon}\right\} \; \;  \text{for} \; \; x \in (-b, b), \quad \tilde{v}_{\varepsilon, 0}(x) = -\frac{1}{\varepsilon} \; \; \text{for} \; \; |x| \ge b. \]
Then, we can also choose $v_{\varepsilon, 0} \in C^\infty(\mathbb{R})$ satisfying 
\[ |\tilde{v}_{\varepsilon, 0} (x) - v_{\varepsilon, 0}(x)| \le \varepsilon \quad \text{for} \; \; x \in [-b, b]. \] 
In order to apply Proposition \ref{prop:upper-esti-smalla}, for $\varepsilon > 0$ small, we choose $v_0$ as $v_{0, \varepsilon} + \varepsilon + b\sqrt{\varepsilon}$ and $\nu > \frac{1}{\varepsilon}$ large so that the constant $T$ obtained in Proposition \ref{prop:upper-esti-smalla} satisfies $T > \frac{1}{\varepsilon}$. 
Let $v_\varepsilon$ and $T_\varepsilon$ be respectively the super-solution and the constant $T$ obtained in Proposition \ref{prop:upper-esti-smalla} in the above settings. 
We hereafter let $v_\varepsilon (x,t) = \infty$ for $x \in (-b, b)$ and $t \ge T_\varepsilon$. 
Since 
\[ - b\sqrt{\varepsilon} \le - \frac{1}{\sqrt{\nu}} \sqrt{r_\nu^2 - x^2} \le 0 \quad \text{for} \; \; |x| \le \frac{2b}{3} \; \text{and} \; \varepsilon > 0  \; \text{small}\] 
if $\nu > \frac{1}{\varepsilon}$, it then holds that 
\begin{align*}
&\lim_{\varepsilon \to 0} v_\varepsilon(x,0) = u_0(x) \quad \text{for} \; \; x \in (-b, b), \\
&u_0(x) \le v_\varepsilon(x, 0) \quad \text{for} \; \; x \in (-b, b), \; \; \varepsilon > 0, \\
&v_{\varepsilon_2}(x,0) \le v_{\varepsilon_1}(x,0) + 2\varepsilon_2 + b\sqrt{\varepsilon_2} \quad \text{for} \; \; x \in (-b, b), \; \; \varepsilon_1 \ge \varepsilon_2 > 0. 
\end{align*}

Let $w_\varepsilon$ be the sequence of sub-solutions constructed in Section \ref{subsec:exists1} when (B2) is assumed and in Section \ref{subsec:exists2} when (B1) assumed. 
Then, a sequence of solution to \eqref{flow eq} can be constructed by 
\[ u_\varepsilon (x,t) :=  \sup\{u(x,t): u \; \text{is a sub-solution to \eqref{flow eq} satisfying} \; w_\varepsilon \le u \le v_\varepsilon\}. \]
Since $T_\varepsilon \to \infty$ as $\varepsilon \to 0$, we can take the relaxed half limit of $(u_\varepsilon)^*$ globally in time as in Section \ref{subsec:exists1} or Section \ref{subsec:exists2}. 
Therefore, due to the argument in Section \ref{subsec:exists1} or Section \ref{subsec:exists2}, we can show that the limit $u$ is a solution to \eqref{flow eq}--\eqref{initial}. 
Since the rest of the proof is same with the proofs in Section \ref{subsec:exists1} and Section \ref{subsec:exists2}, we omit it. 
\end{proof}

\subsection{Non-existence for $\alpha < 1$ and $\beta \ge \frac{1}{1-\alpha}$} \label{sec:non-existence1}

In this section, we prove the non-existence of solution to \eqref{flow eq}--\eqref{initial} if $\alpha < 1$ and $\beta \ge \frac{1}{1-\alpha}$. 
The proof will be completed by construction a sequence of sub-solutions $v_L$ to \eqref{flow eq} which is uniformly bounded at $t=0$ and diverges to $\infty$ as $L \to \infty$ on $(-b, b) \times (0,\infty)$. 
The idea is from the approximated inequality, which follows from the assumptions (A2) and (A3), 
\[ \frac{f(g(V_y(y)) V_{yy}(y))}{V_{y}} \approx L^{\beta + \beta(1 - \alpha) -1} y^{L(1-\beta(1-\alpha)) - \beta} \ge L^{\beta + \beta(1-\alpha) - 1} =: c_L \quad \]
if $0 < y \le 1$ and $L \gg 1$ for $V(y) = y^{-L}$. 
The boundedness from blow holds due to $1 - \beta(1-\alpha) \le 0$. 
Since $c_L \to \infty$ and $V(y) \to \infty$ as $L \to \infty$ for $0 < y < 1$, the sequence of sub-solutions can be constructed as a shifted and scaled version of traveling functions $V(c_Lt -x)$. 
We first define the sequence of sub-solutions as follows: 

\begin{defi}
For constants $L > 0$ and $c_L > 0$, we define $v_L \in C((-b, b) \times [0, \infty))$ as 
\begin{equation}\label{def-vL} 
v_L(x,t) := \begin{cases}
(\frac{3b-x}{2b} - c_L t)^{-L} & \text{for} \quad b-2bc_Lt \le x < b, \; 0 < t < \frac{1}{c_L}, \\
1 & \text{for} \quad -b < x < b-2bc_Lt, \; 0 \le t < \frac{1}{c_L}, \\
(\frac{b-x}{2b})^{-L} & \text{for} \quad -b < x < b, \; t \ge \frac{1}{c_L}. 
\end{cases} 
\end{equation}
\end{defi}

Then, we can prove that $v_L$ is a sub-solution to \eqref{flow eq} for suitable $c_L$ and $L>0$ large as follows: 

\begin{prop}\label{prop:sub-solutions}
Let $b>0, \alpha<1$ and $\beta \ge \frac{1}{1-\alpha}$. 
Assume that function $f$ satisfies (A1) and (A3) and function $g$ satisfies (A2). 
Then, there exist $L_0 > 0$ such that for any $L \ge L_0$ there exists $c_L$ such that 
\begin{equation}\label{div-cL}
\lim_{L \to \infty} c_L = \infty
\end{equation}
and the function $v_L$ defined by \eqref{def-vL} satisfies
\begin{equation}\label{sub-sol-vL}
(v_L)_t \le f\left(\frac{1}{2}g((v_L)_x) (v_L)_{xx}\right) \quad \text{on} \; \; (-b, b) \times (0, \infty)
\end{equation}
in the viscosity sense. 
\end{prop}

\begin{rmk}
Due to \eqref{div-cL}, we can see 
\begin{align*}
&v_L(\cdot, 0) \equiv 1 \quad \text{for} \; \; L > 0, \\
&\lim_{L \to \infty} v_L(x,t) = \infty \quad \text{for} \; \; (x,t) \in (-b, b) \times (0, \infty)
\end{align*}
for $c_L$ obtained in Proposition \ref{prop:sub-solutions}. 
These conditions can be applied to prove (a) in Theorem \ref{thm:non-existence}. 
\end{rmk}

\begin{proof}
We first choose $L_0$ and $c_L$ so that $v_L$ satisfies \eqref{sub-sol-vL} for $b-2bc_Lt < x < b$ and $0 < t < \frac{1}{c_L}$. 
We note that $0 < \frac{3b-x}{2b} - c_Lt < 1$ holds in this case. 
Due to the assumptions (A2) and (A3), there exist constants $L_g, M_g, L_f, M_f > 0$ such that 
\begin{align*}
& g(s) \ge M_g s^{-\alpha} \quad \text{for} \; \; s \ge L_g, \\
& f(s) \ge M_f s^{\beta} \quad \text{for} \; \; s \ge L_f. 
\end{align*}
Define $c_L$ by 
\[ c_L := \frac{M_f M_g^\beta L^{2\beta - \alpha \beta - 1}}{2^\beta(2b)^{(2-\alpha)\beta}}. \]
We note that, due to $2 \beta - \alpha \beta - 1 > 0$, the divergence property \eqref{div-cL} holds. 
Hereafter, we let $b-2bc_Lt < x < b$ and $0 < t < \frac{1}{c_L}$. 
By a simple calculation, we have 
\begin{equation}\label{con-sub1} 
(v_L)_x(x, t) \ge (v_L)_x(b-2bc_Lt, t) = \frac{L}{2b} \ge L_g 
\end{equation}
if $L > 0$ is sufficiently large. 
Therefore, due to $\alpha < 1$, we obtain 
\begin{equation}\label{con-sub2}
\begin{aligned} 
g((v_L)_x(x,t)) (v_L)_{xx}(x,t) \ge&\; M_g ((v_L)_x(x, t))^{-\alpha} (v_L)_{xx}(x, t) \\
=&\; M_g (2b)^{2-\alpha} L^{1-\alpha}(L + 1) \left(\frac{3b-x}{2b} - c_L t\right)^{-L(1-\alpha) - (2- \alpha)} \\
\ge&\; M_g (2b)^{2-\alpha} L^{1-\alpha} (L + 1) \ge 2L_f 
\end{aligned}
\end{equation}
if $L > 0$ is sufficient large. 
Let $L_0 > 0$ be a constant such that \eqref{con-sub1} and \eqref{con-sub2} hold for $L \ge L_0$. 
From $1 - \beta (1-\alpha) \le 0$ and $-2\beta + \alpha \beta + 1 < 0$, for $L \ge L_0$, we have
\begin{align*}
\frac{f(\frac{1}{2}g((v_L)_x(x,t)) (v_L)_{xx}(x,t))}{(v_L)_t (x, t)} \ge&\; \frac{M_f\left(\frac{1}{2}g((v_L)_x(x,t)) (v_L)_{xx}(x,t)\right)^{\beta}}{(v_L)_t(x,t)} \\
\ge&\; \frac{M_f \left(\frac{1}{2}M_g ((v_L)_x(x,t))^{-\alpha} (v_L)_{xx}(x,t)\right)^{\beta}}{(v_L)_t(x,t)} \\
\ge &\; \frac{M_f M_g^\beta L^{2\beta - \alpha \beta -1}}{c_L 2^\beta(2b)^{(2-\alpha)\beta}}\left(\frac{3b - x}{2b} - c_L t \right)^{\delta(1-\beta(1-\alpha)) - 2\beta + \alpha \beta + 1} \\
\ge&\; 1, 
\end{align*}
which yields that $v_L$ satisfies \eqref{sub-sol-vL} if $b-2bc_Lt < x < b$ and $0 < t < \frac{1}{c_L}$ in the classical sense. 
In the other cases, \eqref{sub-sol-vL} obviously holds in the viscosity sense, so we omit the details. 
\end{proof}

Let us now prove Theorem \ref{thm:non-existence} (a). 

\begin{proof}[Proof of Theorem \ref{thm:non-existence} (a)] 
Let $M := \inf_{x \in (-b,b)} u_0(x)$. 
Assume that a solution $u$ to \eqref{flow eq}--\eqref{initial} exists. 
Let $v_L$ be the sequence of sub-solutions to \eqref{flow eq} obtained in Proposition \ref{prop:sub-solutions}. 
Then, from $v_L(\cdot, 0) - (M+1) \le u_0$ and the convexity of $v_L(\cdot, t)$, we have by the comparison result in Theorem \ref{thm:com-conti2} 
\[ v_L(x, t) - (M+1) \le u(x,t) \quad \text{for} \; (x,t) \in (-b,b) \times (0,1/c]. \]
Letting $L \to \infty$, we have $u(x, t) = \infty$ for $(x,t) \in (-b, b) \times (0, \infty)$, which contradicts that $u$ is a function defined on $(-b, b) \times [0, \infty)$. 
\end{proof}

\appendix

\section{Traveling wave solutions}

In this section, we discuss properties of traveling wave solutions to \eqref{singular-neu} formed by $w(x,t) = W(x) + ct$. 
Therefore, the profile equation is given by 
\begin{equation}\label{eq profile}
\begin{cases}
f^{-1}(c) = g(W_x)W_{xx} & \text{for} \; \; -b < x < b, \\
\lim_{x \to \pm b} W_x(x) = \pm \infty. 
\end{cases}
\end{equation}
The existence and uniqueness of the traveling wave solutions in the viscosity sense was studied in \cite{KL} when $\alpha > 2$, where $\alpha$ is the constant in the assumption (A2). 
We here study classical solutions when $\alpha \le 2$, namely, we assume that $W$ is sufficiently smooth in $(-b, b)$. 

\begin{prop}\label{prop:exists-tw}
Let $b>0$. 
Assume that $f, g \in C(\mathbb{R})$ satisfy (A1) and (A2). 
\begin{itemize}
\item[(a)] If $1< \alpha \le 2$, then there exists a pair $(W, c) \in C^2((-b, b)) \times (0, \infty)$ satisfying \eqref{eq profile}. Moreover, if $(\tilde{W}, \tilde{c})$ is another pair satisfying \eqref{eq profile}, then there exists $\alpha \in \mathbb{R}$ such that $\tilde{W} = W$ and $\tilde{c} = c$. Furthermore, $W$ satisfies 
\begin{equation}\label{divergence-W}
\lim_{x \to \pm b} W(x) = \infty
\end{equation}
and there exist constants $D_\pm > 0$ such that 
\begin{equation}\label{div-rate-W} 
\lim_{x \to \pm b} W(x) \left(\psi_{\frac{2-\alpha}{\alpha -1}}(b\mp x)\right)^{-1} = D_\pm, 
\end{equation}
where $\psi_{\frac{2-\alpha}{\alpha -1}}$ is the function defined by \eqref{def-psi}. 
\item[(b)] If $\alpha \le 1$, then there is no pair $(W, c) \in C^2((-b, b)) \times (0, \infty)$ satisfying \eqref{eq profile}.
\end{itemize}
\end{prop}

\begin{proof}
We first prove (a). 
We will construct $c$ and $W$ satisfying \eqref{eq profile}. 
When $1 < \alpha \le 2$, due to the assumption (A2), the function 
\[ G(s) := \int_{-\infty}^s g(\tilde{s}) \; d\tilde{s} \]
is well-defined. 
Integrating the first equality in \eqref{eq profile} on $(-b, b)$, we have 
\[ 2bf^{-1}(c) = \int_{-b}^b g(W_x(x))W_{xx}(x) \; dx = \int_{-\infty}^\infty g(s) \; ds < \infty, \]
which yields 
\[ c = f\left( \frac{G(\infty)}{2b} \right). \]
We can also obtain by integrating the first equality in \eqref{eq profile} on $(-b, x]$ 
\[ G(W_x(x)) = \int_{-b}^x g(W_x(\tilde{x})) W_{xx}(\tilde{x}) \; dx = (x+b) f^{-1}(c). \]
Since $G$ is strictly increasing, the inverse function $G^{-1}$ is well-defined. 
By choosing $W(0) \in \mathbb{R}$ arbitrary, we thus obtain 
\[ W(x) = \int_0^x W_x(\tilde{x}) \; dx +W(0) = \int_0^x G^{-1}((\tilde{x} + b) f^{-1}(c)) \; d\tilde{x} + W(0). \]
By the construction of $c$ and $W$, the pair $(W, c)$ obviously satisfies the first equation of \eqref{eq profile}. 
Due to 
\begin{align*} 
&W_x(x) = G^{-1}((x + b) f^{-1}(c)) = G^{-1} \left( \frac{(x+b) G(\infty)}{2b}\right), \\ 
&\lim_{s \downarrow 0} G^{-1}(s) = -\infty, \quad \lim_{s \uparrow G(\infty)} G^{-1} (s) = \infty, 
\end{align*}
we can see that the boundary condition in \eqref{eq profile} holds. 
The unique in the sense described in (a) also obviously holds by the construction of $c$ and $W$. 

The divergence property \eqref{divergence-W} and its divergence rate \eqref{div-rate-W} can be proved by l'H\^{o}pital's rule. 
Indeed, due to l'H\^{o}pital's rule, we have by the assumption (A2) 
\[ \lim_{s \to \infty} \frac{2bf^{-1}(c) - G(s)}{s^{-\alpha+1}} = \lim_{s \to \infty} \frac{-g(s)}{(1-\alpha)s^{-\alpha}} = \frac{C_+}{\alpha-1}, \]
where $C_+$ is the constant in the assumption (A2). 
Since $W_x(x) \to \infty$ as $x \to b$, changing the variables as $s = W_x(x)$ yields 
\[ \lim_{x \to b} \frac{f^{-1}(c) (b-x)}{(W_x(x))^{-\alpha + 1}} = \frac{C_+}{\alpha -1}, \]
which is equivalent to 
\[ \lim_{x \to b} \frac{W_x(x)}{(b-x)^{-\frac{1}{\alpha -1}}} = \left(\frac{f^{-1}(c) (\alpha -1)}{C_+} \right)^{-\frac{1}{\alpha -1}}. \]
Therefore, since $-\frac{1}{\alpha-1} \le 1$, we can see that \eqref{divergence-W} holds at $x=b$. 
We then can apply l'H\^{o}pital's theorem again to obtain \eqref{div-rate-W} at $x=b$ due to $-\frac{2-\alpha}{\alpha-1} - 1 = -\frac{1}{\alpha-1}$. 
At the opposite point $x=-b$, a similar argument can be applied to obtain \eqref{divergence-W} and \eqref{div-rate-W}. 

We next prove (b). 
Assume the existence of a pair $(c, W)$ satisfying \eqref{eq profile} for contradiction. 
Integration of the first equation of \eqref{eq profile} on $(-b,b)$ shows by the assumption (A2) with $\alpha \le 1$ 
\[ f^{-1}(c) = \int_{-b}^b g(W_x(x)) W_{xx}(x) \; dx = \int_{-\infty}^\infty g(s) \; ds = \infty, \]
which contradicts to $c \in \mathbb{R}$. 
\end{proof}

\bigskip

\noindent
{\bf Data availability}: There is no conflict of interest. 



\end{document}